\documentclass[reqno,11pt]{amsart}
\usepackage{amsthm,amsfonts,amssymb,euscript,mathrsfs,graphics,color,amsmath,amssymb,latexsym,marginnote}%,mathabx}
\usepackage[dvips]{graphicx}
\usepackage[margin=0.9in]{geometry}

\usepackage[notcite,notref]{showkeys}
\usepackage{hyperref}

\theoremstyle{plain}
\newtheorem{theorem}{Theorem}
\newtheorem{definition}[theorem]{Definition}
\newtheorem{remark}[theorem]{Remark}
\newtheorem{proposition}[theorem]{Proposition}
\newtheorem{lemma}[theorem]{Lemma}

\numberwithin{equation}{section} 
\numberwithin{theorem}{section}

\newcommand{\mf}{\widetilde}

\def\wtF{\widetilde{\mathcal{F}}}
\def\whF{\widehat{\mathcal{F}}}

\def\supp{\mathrm{supp}}
\def\e{{\varepsilon}}

\def\eps{{\epsilon}}
\def\epss{{\epsilon_1,\epsilon_2}}

\def\Z{{\mathbb Z}}

\def\R{{\mathbb R}}

\def\jxi{\langle \xi \rangle}
\def\jeta{\langle \eta \rangle}
\def\jsigma{\langle \sigma \rangle}
\def\jsig{\langle \sigma \rangle}

\def\jt{\langle t \rangle}
\def\js{\langle s \rangle}
\def\jx{\langle x \rangle}

\def\Ftil{\wt{\mathcal{F}}}

\def\s{\sigma}

\def\jnab{\langle \nabla \rangle}

\def\pv{\mathrm{p.v.}}

\def\wt{\widetilde}

\definecolor{bluegreen}{rgb}{0.0, 0.3, 0.9}

\begin{document}

\author{Tristan L\'eger}
\address{Tristan L\'eger, Princeton University, Department of Mathematics, Fine Hall, %Rm 6290, 
  Washington Road, Princeton, NJ, 08544, United States.}
\email{tleger@princeton.edu}

\author{Fabio Pusateri}
\address{Fabio Pusateri, University of Toronto, Department of Mathematics, 40 St George Street, %Rm 6290, 
  Toronto, ON, M5S 2E4, Canada.}
\email{fabiop@math.toronto.edu}

%\title[Quadratic Klein-Gordon with bound state]{Quadratic Klein-Gordon with bound state}
\title[Internal modes for quadratic KG]{Internal mode-induced growth \\ in $3$d nonlinear Klein-Gordon equations}

\begin{abstract}

\small
This note complements the paper \cite{LP}
by proving a scattering statement for solutions of nonlinear Klein-Gordon equations 
with an internal mode in $3$d.
We show that small solutions %to a Klein-Gordon equation with an internal mode 
exhibit growth around a one-dimensional set in frequency space and become of order one in $L^\infty$ after a short transient time.
%that the profile as consequence, the amplitude of 
The dynamics are driven by the feedback of the internal mode into the equation for the field (continuous spectral)
component. %of the solutions. %(the ``bad'' part identified in \cite{LP}).

The main part of the proof consists of showing suitable smallness for a ``good'' component of the radiation field. 
This is done in two steps: first, using the machinery developed in \cite{LP}, we reduce the problem to bounding 
a certain quadratic normal form correction.  
Then we control this latter %by relying on new
by establishing some refined estimates for certain bilinear operators with singular kernels.

\end{abstract}

\thanks{F.P. was supported in part by a start-up grant from the University of Toronto, 
and NSERC grant RGPIN-2018-06487. %, and a Connaught Foundation New Researcher Award.
}

\maketitle

\tableofcontents

\section{Introduction}

\subsection{Background}
The study of the small data regime for dispersive PDEs with quadratic nonlinearities 
has attracted a lot of attention in recent years, and is closely related to several open
questions about the asymptotic stability of coherent structures, such as (topological) solitons. 
This has turned out to be a challenging problem, since quadratic nonlinearities 
are strong enough to influence the long-time behavior of solutions and may lead to blow-up
even in a weakly nonlinear regime.
%In particular, different quadratic nonlinearities can lead to different asymptotics
%\fp{as is well known since the works of John and Klainerman \cite{John,K} on wave equations}
%
%For the Klein-Gordon equation in the flat case, i.e., 
%without non-homogeneous or linear potential terms, %constant coefficients \cfp{Not really...}) 
%general quasilinear quadratic nonlinearities were considered %in 1985 
%by Shatah \cite{shatahKGE} and Klainerman \cite{KKGE}.
%who simultaneously developed the normal  form method and vector field method respectively. 
%
%Around 2009 these two techniques were in some
%sense unified under the space-time resonance framework, which was first implemented for the 
%quadratic NLS equation in 3d \cite{GMS}. 
%
%This new approach, along with its subsequent refinements,
%has proved extremely effective to construct global decaying solutions to many dispersive equations. 
%

In recent years, several general methods have been developed to tackle
problems about global existence and long-time asymptotics of nonlinear (and quasilinear) dispersive equations
with low power nonlinearities.
Without trying to be exhaustive, we mention %works on the water waves equation in 3d 
works by Germain-Masmoudi-Shatah 
%on nonlinear Schr\"odinger equations 
\cite{GMS} and %$3$d water waves 
\cite{GMSww}
and the development of the `space-time resonance' method, %\cite{GMSww2}, %Gustafson Nakanishi-Tsai 
parallel work by Gustafson-Nakanishi-Tsai \cite{GNT} on Gross-Pitaevski,
%Ionescu and the second author were able to treat the water waves system in 2d (\cite{IPu1}, \cite{IPu2}).
Guo-Ionescu-Pausader \cite{GIP} %as well as the Euler-Maxwell system for electrons in 2d 
and Deng-Ionescu-Pausader \cite{DIP} on the Euler-Maxwell system,
Ionescu and the second author on $2$d water waves \cite{IPu1,IPu2}, 
%They also studied, with Y. Deng and B. Pausader the gravity-capillary water wave system in three dimensions in 
and Deng-Ionescu-Pausader and the second author \cite{DIPP} on $3$d gravity-capillary waves.
%Among other equations that were considered, let us mention the Euler-Poisson system in 2d in 
%and Ionescu-Pausader \cite{IP}, 
%the Euler-Maxwell system in 3d 
All of these results can be interpreted as proving asymptotic stability of the zero or `trivial' solution. 

To prove similar results for non-trivial solutions, 
it is natural to consider corresponding linearized problems. 
This generically introduces a potential term in the equation.  
A typical model example is $i \partial_t u + L(\sqrt{-\Delta + V + m^2}) = u^2,$ for some 
real-valued dispersion relation $L$ and potential $V=V(x)$. 
Consequently, a new line of active research has emerged
seeking to understand the global behavior %asymptotics for 
of small solutions to such equations.

A first result was obtained for the quadratic NLS 
$i\partial_t u +(-\Delta+V)u = \bar{u}^2$  in $3$d by Germain-Hani-Walsh in \cite{GHW}. 
%In the course of the proof, the authors developed some harmonic analytic tools adapted to the non-flat background,
%and proved some Coifman-Meyer type multilinear bounds for operators \cfp{complete}
%with the NSD $\mu$ (see \eqref{intromu}), 
%as well as a commutation identities for the weight $|x|$.
The first author of this note dealt with the case of a $u^2$ nonlinearity with small, time-dependent $V$ \cite{L}.
Note that while a $\bar{u}^2$ nonlinearity is non-resonant (for any type of evolution), a $u^2$ 
nonlinearity is space-time resonant at the origin even in the case $V=0$.
See also \cite{Leger2} for the case of electromagnetic perturbations. 
More recently, Soffer and the second author \cite{PS} treated a $u^2$ nonlinearity with a large potential,
and developed a general approach based on the distorted Fourier transform,
that is, the transform adapted to the Schr\"odinger operator $H:=-\Delta + V(x)$.
The starting point of this is a refined study of 
%continued the study of 
the so-called `nonlinear spectral distribution' (NSD)
\begin{align}\label{intromu}
\mu(\xi,\eta,\sigma) := (2 \pi)^{-9/2} \int_{\R^3} \overline{\psi(x,\xi)}
  \psi(x,\eta) \psi(x,\sigma) dx, 
\end{align}
where $\psi=\psi(x,\xi)$, $\xi \in \R^3$ are the generalized eigenfunctions of the Schr\"odinger operator,
$(-\Delta + V(x)) \psi = |\xi|^2 \psi$.
%Their work provides an expansion of $\mu,$ extracting its singular parts, 
%and proves product estimates for the operators associated to the various terms present
%(see \eqref{intromu'}). This refined understanding of the regularity of $\mu$ 
%allows the full power of the space-time resonance method to be used, 
%in the sense that oscillations can be exploited in frequency space.
In particular, a precise knowledge of the structure and singularities of $\mu$ 
allows the exploitation of oscillations in (distorted) frequency space,
which is a key step to extend the techniques from the works on the case $V=0$ cited above.
%this so that one can approach the study of localized perturbations of large coherent structures.

We also mention that there have been several recent works in $1$d on 
equations with potentials and low power nonlinearities.
Since this is not the main focus of this note, we refer the reader to the introductions
of \cite{GPKG} and \cite{LSSineG} for extensive discussions and references.
%see for example \cite{DelortNLSV,GPR} on NLS, \cite{LLSS,GPKG} on Klein-Gordon and references therein,
%as well as works on the stability of 

\smallskip
Up until now, global dispersive solutions for low powers (in particular quadratic) 
in\footnote{In $1$d we mention important recent works \cite{KowMarMun} and \cite{DMKink} on the 
stability under odd perturbations of the kink for the $\phi^4$ model. 
The linearization at the kink gives exactly a $1$d quadratic Klein-Gordon equation with an internal mode like 
\eqref{introKG}-\eqref{introim} satisfying the Fermi Golden rule \eqref{introFGR}.
} 
$d\leqslant 3$ had only been constructed in the case of `trivial' spectrum, 
$\sigma( -\Delta + V)=\sigma(-\Delta)$. 
In \cite{LP} we dealt with a problem where an {\it internal mode} is present.
More precisely, we considered the following initial value problem:
\begin{equation}\label{introKG}
\begin{cases}
\partial_t ^2 u + (-\Delta + V(x) + 1)u = u^2, 
\\
u(0,x) = u_0, \quad \partial_t u(0,x) = u_1,
\end{cases}
\end{equation}
where $u: \R_t \times \R^3_x \rightarrow \R$,
with a sufficiently regular and decaying generic
external potential $V:\R^3\rightarrow \R$, 
and smooth, localized initial data $u_0$ and $u_1$. 
We will make the assumptions on $V$ and the initial data more precise below.

%we shall do this later but keep the mass parameter at the beginning of this introduction.
%but this can be obviously rescaled to any $m > 0$.
%We assume that $0$ energy is regular for $H := -\Delta+V$ 
%(see below for the exact definition and comments on this assumptions).

In this setting, the spectrum of the Schr\"odinger operator $H:= -\Delta + V$ consists of the purely absolutely continuous part $[0,\infty)$ and a finite number of negative eigenvalues, 
with corresponding smooth and fast decaying eigenfunctions \cite{SimonSpec}.
We further assume that 
%$H$ has only one negative eigenvalue as follows:
the operator $L^2 := -\Delta + V + 1$ has a unique strictly positive eigenvalue
%, see Remark \ref{remmulti}.}
with a corresponding normalized eigenfunction $\phi$: %cfp{positive?}
\begin{align}\label{introim}
(-\Delta + V + 1)\phi = \lambda^2 \phi, \qquad  1/2 < \lambda < 1, \qquad {\| \phi \|}_{L^2} = 1.
\end{align}
The eigenvalue $\lambda$ is usually called the ``internal frequency of oscillation'',
and $\phi$ is referred to as an ``internal mode'' of the dynamics.

\eqref{introim} gives rise to a two parameter family of solutions to the linear equation 
$\partial_t ^2 u + (-\Delta + V + 1)u=0$ of the form:
\begin{align}\label{phi0}
\phi_{A,\theta} (t,x) = A \cos(\lambda t + \theta) \phi(x), \qquad A,\theta \in \R.
\end{align}
These solutions - referred to as ``bound states'', or ``internal modes'' with a slight abuse -
are time-periodic, oscillating and spatially localized. 
In this weakly nonlinear regime, we expect nonlinear solutions to retain features of the linear system. 
Therefore, a natural question is whether such periodic solutions persist under the nonlinear flow. 
Our main result in \cite{LP} answered in the negative: 
in a neighborhood of zero, the bound states \eqref{phi0} are destroyed by the quadratic nonlinearity,
and do not continue to quasi-periodic or other non-decaying solutions. 
Moreover, solutions exhibit weak dispersive features and obey decay estimates,
although with a rate that is much slower compared to solutions of $(\partial_t^2 - \Delta + 1 + V)u=0.$ 

%\smallskip
%\noindent
%{\it Some related works.} 
The first result of this kind was obtained by Sigal \cite{Sigal}, who showed instability
of bound states for very general classes of equations, introducing the ``nonlinear Fermi Golden Rule'' (FGR). 
%for general Hamiltonian,
Soffer and Weinstein \cite{SWmain} then derived decay and asymptotics
in the case of cubic Klein-Gordon equations, that is, \eqref{introKG} with a $u^3$ nonlinearity. 
%It is important to note that, %one expects
%Let us stress that 
%the case of a quadratic nonlinearity is substantially more difficult
%than the cubic case, as it is to be expected in the context of the long-time analysis 
%of dispersive and wave equations.  
See also the recent work \cite{LeiLiuYang} on the case $\lambda \in (1/3,1/2)$.
The problem of meta-stability in the presence of multiple bound states %and in the energy space 
in $3$d was treated by Bambusi-Cuccagna \cite{Bambusi-Cuccagna}.
For the nonlinear Schr\"odinger equation (NLS) the problem of meta-stability (for excited states) 
was studied by Tsai-Yau \cite{TsaiYau}; see also the recent advances by Cuccagna-Maeda \cite{CuMa2},
their survey \cite{CuMaSurII} and references therein. 

Our main result in this note complements \cite{LP} with precise scattering statements and asymptotics. 
%pointwise in (distorted) frequency space.
In particular, we will show scattering in $L^\infty_\xi$, i.e., pointwise in (distorted) frequency space,
for the `good' component of the radiation identified in \cite{LP}.
As a consequence we will obtain that solutions of \eqref{introKG} become of $O(1)$ in $L^\infty_\xi$
after a short time, and exhibit a certain oscillatory behavior over time.
%, and are periodically in time of $O(1)$.
This appears to be the first rigorous proof for this type of asymptotic dynamics.

\subsection{The main result of \cite{LP}}% and some related works}
%We recall the main result of \cite{LP}. Let us first introduce some notation and state our assumptions.
%We denote $L^2 = -\Delta + V + 1.$  

\smallskip
\noindent
{\it Set-up and assumptions.} 
Let $\bf{P_c}$ denote the projection onto the continuous spectral subspace, 
namely, for every $\psi \in L^2(\R^3;\R)$
\begin{align}\label{proj}
{\bf P_c} \psi := \psi - (\phi,\psi)\phi, \qquad (\psi_1,\psi_2) := \int_{\R^3} \psi_1 \psi_2 \, dx.
\end{align}
We make the following assumptions:

\begin{itemize}

\medskip
\item {\em Regularity and decay of $V$}: We assume $V \in \mathcal{S}$, but a finite amount of regularity and decay (measured in weighted Sobolev spaces) would suffice.

\medskip
\item {\em Coupling to continuous spectrum}: $\lambda \in (1/2,1)$ so that $\lambda \notin \sigma_{ac}(-\Delta + V + 1)$ 
  and $2 \lambda \in \sigma_{ac}(-\Delta + V + 1).$ 
  
\medskip
\item {\em Fermi Golden Rule}: The ``Fermi Golden Rule'' resonance condition holds:
\begin{align}\label{introFGR}
\Gamma:=\frac{\pi}{2 \lambda} \Big({\bf{P_c}} \phi^2, \delta(L-2\lambda) {\bf P_c} \phi^2 \Big) > 0,
\end{align}
where $L:= \sqrt{-\Delta+V+1}$.

\medskip
\item {\em Genericity of $V$}: the $0$ energy level is regular for $H := -\Delta+V$, 
that is, $0$ is not an eigenvalue, 
nor a resonance, i.e., there is no $\psi \in \jx^{1/2+} L^2(\R^3)$ such that $H\psi = 0$. 
Such a potential $V$ is said to be `generic'.

\end{itemize}

\smallskip
%\noindent
{\it Distorted Fourier transform.} 
%We define the distorted Fourier transform. 
For a regular and decaying $V$ (as in our assumptions above)
we can decompose $L^2$ into absolutely continuous and pure point subspaces:
$L^2(\R^d) = L^2_{ac}(\R^d) \oplus L^2 _{pp}(\R^d)$ where $L^2 _{pp}(\mathbb{R}^d) = \mathrm{span}(\phi)$. 
Let $\psi(x,\xi)$ denote the generalized eigenfunctions %(or modified complex exponential) 
solving for $\xi,x \in \R^3$
\begin{align}\label{psieq0}
\begin{split}
& (-\Delta + V ) \psi(x,\xi) = \vert \xi \vert^2 \psi(x,\xi), %\qquad \xi,x \in \R^3,
  %\\
  %& 
  \quad \mbox{with} \quad \big| \psi(x,\xi) - e^{ix\cdot \xi} \big| \longrightarrow 0,
  \quad \mbox{as} \quad |x|\rightarrow \infty,
\end{split}
\end{align}
%see \eqref{assV1} (or if one restricts to the the continuous spectrum)
with the Sommerfeld radiation condition $r (\partial_r - i|k|) (\psi(x,k) - e^{ix\cdot k}) 
\rightarrow 0$, 
for $r=|x| \rightarrow \infty$.
We can then define a unitary operator $\Ftil$, the distorted Fourier Transform (dFT), as 
\begin{align}\label{Ftildef}
\wt{\mathcal{F}}f(\xi) := \wt{f}(\xi) = \frac{1}{(2\pi)^{3/2}} \lim_{R \to + \infty} 
\int_{\vert x \vert \leqslant R} f(x) \overline{\psi(x,\xi)} dx,
\end{align}
with inverse 
\begin{align}\label{Ftilinvdef}
\big(\wt{\mathcal{F}}^{-1}f\big)(x) = \frac{1}{(2\pi)^{3/2}} \lim_{R \to + \infty} 
\int_{\vert \xi \vert \leqslant R} f(\xi) \psi(x,\xi) d\xi.
\end{align}
In particular, for any $g\in L^2$ we can write
%\begin{align}\label{dFT}
%\begin{split}
%& \wt{\mathcal{F}}g(\xi) := \widetilde{g}(\xi) = \frac{1}{(2\pi)^{3/2}}
%  \int_{\R^3} \overline{\psi(x,\xi)} g(x)\, \mathrm{d}x, \qquad 
%  \\ & \mbox{with} 
%  \qquad \Ftil^{-1}g (x) := \frac{1}{(2\pi)^{3/2}}\int_{\R^3} {\psi(x,\xi)} g(\xi) \,\mathrm{d}\xi,
%\end{split}
%\end{align}
%\begin{align}\label{dFTinvintro}
%\begin{split} & 
$g(x) = \wtF^{-1} \big( \wt{{\bf P}_c g} \big) + (g, \phi) \phi$.
%\end{split}
%\end{align}
Moreover $\wtF$ diagonalizes the Schr\"odinger operator (on the continuous spectrum), $\Ftil {\bf{P_c}} H = |\xi|^2 \Ftil$.

\smallskip
%\noindent
{\it Continuous-Discrete decomposition.} 
Next, we decompose a solution $u$ into a discrete and a continuous component:
\begin{align}\label{uav}
u(t) = a(t) \phi + v(t), \quad \mbox{with \, $(v(t),\phi)=0$ \, for all $t.$}
\end{align} 
Then \eqref{introKG} with \eqref{introim} reads
\begin{equation}\label{sysav}
\begin{cases}
\ddot{a} + \lambda^2 a = \big((a \phi + v)^2 ,\phi\big) 
\\
\partial_t ^2 v + L^2 v = {\bf{P_c}} \big((a \phi + v)^2 \big), \qquad L := \sqrt{-\Delta + V + 1}.
\end{cases}
\end{equation}
We call $v$ the {\it radiation} or {\it field}
component of the solution $u$, and will call $a$ the (amplitude of the) 
{\it discrete component} or {\it internal mode} of the solution.

\smallskip
%\noindent
{\it Main result on Radiation Damping.} 

\begin{theorem}[\cite{LP}]\label{maintheo}
Consider \eqref{introKG}, with initial data $u(0,x) = u_0(x)$, $u_t(0,x)=u_1(x)$ such that: 
\begin{align}\label{mtdata}
\begin{split}
%u_0 = a(0)\phi + v_0,
& | ( u_0, \phi) | + | ( u_1, \phi) | \leqslant \e_0,
\\
& {\|  (\jnab u_0, u_1) \|}_{W^{3,1}}+  {\| \jx (\jnab u_0, u_1) \|}_{H^3} + {\| (\jnab u_0, u_1) \|}_{H^N}  
  \leqslant \e_0 
%\qquad {\| \jx^2 u_1 \|}_{L^2} + {\| u_1 \|}_{H^N}
\end{split}
\end{align}
for $N \gg 1$.
Then, there exists $\overline{\e} \in (0,1)$ such that, for all $\e_0 \leq \bar{\e}$,

\setlength{\leftmargini}{1.5em}
\begin{itemize}

\smallskip
\item The equation \eqref{introKG} under the assumptions stated above has a unique global solution 
$u \in C(\R; H^{N+1}(\R^3))$ such that the following hold:
$u = a(t) \phi  + v(t)$ with
\begin{align}
\label{mtadecay}
& |a(t)| \approx \e_0 (1+\e_0^2 t)^{-1/2}, 
\end{align}
and
\begin{align}
\label{mtvdecay}
& \jt {\big\| \partial_t v + i L v \big\|}_{L^\infty_x} \approx 1, 
  \qquad t \gtrsim \e_0^{-2}, %\qquad p>6. %\e^\beta \jt^{\delta},
\\
\label{mtSobolev}
& {\big\|\partial_t v + i L v \big\|}_{H^N_x} \lesssim  \e_0^{1-\delta},
\end{align}
for arbitrarily small $\delta > 0$.

\smallskip
\item
Furthermore, we have the following asymptotic behavior: 
Let %\begin{align*}%\label{mtprof}
$f:= e^{-itL}(\partial_t + i L)v.$
%\end{align*}
Then, there exists $f_\infty \in \e_0^{1-\delta} H^N_x$, with $\delta >0$ arbitrarily small as above,
such that
%$f(t)$ has a limit $f_{\infty}$ as $t \rightarrow +\infty$ in $H^N,$ 
\begin{align}\label{mtscattHN}
{\| f(t) - f_\infty \|}_{H^N_x} \lesssim \e_0^{1-\delta} \jt^{-\delta'}
\end{align}
for some small $\delta' >0$.
%fp{which satisfies ${\Vert f_{\infty} \Vert}_{H^N} = O(\varepsilon_0^{1-\delta})$, for some $\delta >0$.}
%Then the following global bounds hold
%\begin{align}\label{fasyh}
%{\big\| \partial_\xi \wt{h}(t) \big\|}_{L^2_x} \lesssim \e_0^\beta \jt^{\beta} \quad (t \gtrsim \e_0^{-2}),
%\qquad {\| h(t) \|}_{H^N_x} \lesssim \e_0,
%\end{align}
%for some sufficiently small $\beta>0$.
%and
Finally, we have the asymptotic growth
\begin{align}\label{mtgrowth}
{\big\| \partial_\xi \wt{f}(t) \big\|}_{L^2} \gtrsim \jt^{1/2}, \qquad t \gtrsim \e_0^{-2}.
\end{align}

\end{itemize}

\end{theorem}

\smallskip
%\noindent
{\it Set-up and notation.} 
In this note we work on the solutions constructed in Theorem \ref{maintheo},
and will use notation which is consistent with \cite{LP}.
The starting point for our analysis is the natural definition of a `good' component, $h$,
and a `bad' component, $g$, of the profile $f = e^{-itL}(\partial_t + i L)v$:
\begin{align}\label{defg}
\begin{split}
& f = h - g, \qquad \wt{g}(t,\xi) := -%\varphi_{\leqslant -C}(\jxi - 2 \lambda)
  \chi_C(\xi)
  \int_{0}^t B^2(s) e^{-is(\jxi-2\lambda)} ds \, \wt{{\bf P}_c\phi^2}(\xi),
  %\quad \chi_C(\xi):=\varphi_{\leqslant -C}(\jxi - 2 \lambda)
\end{split}
\end{align}
where $\chi_C(\xi) := \varphi_{\leqslant -C}(\jxi - 2 \lambda)$, 
with $\varphi_{\leqslant -C}(x)$ a bump function supported on $|x| \leqslant 2^{-C+1}$ for $C=C(\lambda)$
large enough,
and where $B=B(s)$ can be defined through the profile of the amplitude of the internal oscillations given by
\begin{align}\label{defA}
A(t) := \frac{1}{2i\lambda}e^{-i\lambda t} (\dot{a} + i \lambda a).
\end{align}
We refer the reader to Lemma 3.5 in \cite{LP} for the details of the definition of $B$ 
which are not necessary here, and only recall the fact that (see \eqref{defrhoZ} below for the definition of $\rho$)
\begin{align}\label{A-B}
\begin{split}
& |A(s)-B(s)| = O(|A(s)|^2) \approx \rho(t),
\\
& \dot{B}(t) = \rho(t)^{3/2} + \rho(t)^{3/2-a} \jt^{-1}, \quad a\in (1/2,1). 
\end{split}
\end{align}
We also define two useful quantities
\begin{align}\label{defrhoZ}
\rho(t):= \varepsilon_0 \big(1 + \varepsilon_0^2 (\Gamma/\lambda) t\big)^{-1}, 
  \qquad
Z := Z_\beta(\e,m) := \big(2^m \rho(2^m)\big)^{1-\beta} \big(2^m \varepsilon \big)^{\beta}
\end{align}
for some absolute constant $0<\beta \ll 1$, and $\varepsilon:=C_0 \varepsilon_0$ for some large constant $C_0>0.$ 
%$\lambda$ and $\Gamma$ are defined in \eqref{introim} and \eqref{introFGR} respectively.

\subsection{New result: Scattering in $L^{\infty}_{\xi}$}
As already mentioned,
our goal is to provide additional and more precise 
information on the asymptotic behavior of solution of \eqref{introKG}.
In particular, we identify a growth phenomenon due to the presence of the internal mode,
and describe the asymptotic behavior of the solution in $L^{\infty}_{\xi}.$ 
This allows us to see that after a short time the solution becomes of size $O(1)$ 
around a certain frequency, 
despite arising from small initial data. Our main result can be stated as follows:

\begin{proposition}\label{main}
There exist $F_{\infty} \in \varepsilon^{\delta} L^{\infty}_t L^{\infty}_{\xi}$
and a real-valued $\Psi_\infty = O(\varepsilon^{\delta})$, with some %arbitrarily 
small $\delta>0$, such that
\begin{align}\label{fasy}
\begin{split}
\wt{f}(t,\xi) = \int_0^t \exp \Big(is(\jxi-2\lambda)
  + i \tfrac{2c_2\lambda}{\Gamma}\log\big(1 + \tfrac{\Gamma}{\lambda} Y_0^2 s\big) \Big) y(s)^2 ds \, e^{i 2\Psi_\infty} 
  \, \wt{{\bf P}_c\phi^2}(\xi)
  \\
  + F_\infty(t,\xi) %\, \wt{h_\infty}(t,\xi) \, 
  + \, O_{L^\infty_\xi}\big(\varepsilon^\delta \jt^{-\delta}\big),
\end{split}  
\end{align}
where $c_2$ denotes a non-zero numerical constant,
and $y(s) := Y_0 \, (1 + (\Gamma/\lambda) s Y^2_0)^{-1/2}$
%\begin{align}
%y(s) := \frac{Y_0}{\sqrt{1 + \tfrac{\Gamma}{\lambda} s Y^2_0}},
%\end{align}
with $Y_0=O(\e_0)$ a function of the initial data $(a(0),\dot{a}(0))$.

In particular, for $t \gtrsim \e_0^{-2}$
and $|\jxi - 2\lambda| \leqslant \delta_0 t^{-1}$ where $\delta_0$ denotes a sufficiently small constant, we have
\begin{align} \label{O(1)}
\wt{f}(t,\xi)&=  \frac{1}{2ic_2} \Big[e^{2ic_2 \frac{\lambda}{\Gamma} \log \big(1 + \frac{\Gamma}{\lambda} Y_0^2 t \big)} 
  - 1 \Big] \, e^{i 2\Psi_\infty} 
  \, \wt{{\bf P}_c\phi^2}(\xi) + O_{L^{\infty}_{\xi}}(\delta_0).
\end{align}
\end{proposition}

\begin{remark}
The leading term in \eqref{O(1)} exhibits an oscillatory behavior, 
and is $O(1)$ for most times. Note that this does not contradict \eqref{mtscattHN}
since this only happens around a one-dimensional set. 
\end{remark}

\begin{remark}
Although the equation we study is quadratic, we expect the same behavior to take place in many other settings, 
typically when an internal mode is present. %since this is what drives the growth.
%Indeed the growth is driven by the feedback of the internal mode into the equation for the field.
\end{remark}

We can reduce the proof of the main Proposition \ref{main} to a scattering statement for the field $h$
plus a quadratic (normal form) correction through the following:

\begin{lemma}\label{Nexist}
Assume there exists a correction $\mathcal{N}(\wt{f}) \in \varepsilon^{\delta} L^{\infty}_t  L^{\infty}_{\xi}$
and $h_\infty \in \e^\delta L^\infty_\xi$ such that
%\cfp{Fix $\mathcal{N}(\wt{h})(t)$ to $f$}
\begin{align} \label{hinfty}
{\big\| \wt{h}(t) - \mathcal{N}(\wt{f})(t) - h_\infty \big\|}_{L^\infty_\xi} \lesssim \e^\delta \jt^{-\delta}
\end{align}
for some $\delta >0$.
Then \eqref{fasy} holds true.
\end{lemma}

\smallskip
\begin{proof}[Proof of Proposition \ref{main} using Lemma \ref{Nexist}]
Recall $(3.64)$ in \cite{LP}:
\begin{align}\label{gasy}
\begin{split}
%& \wt{h} := \wt{f} + \wt{g}, \\
\wt{g}(t,\xi) %& = - \chi_C(\xi) \int_0^t B^2(s) e^{-is(\jxi-2\lambda)} ds \,  \wt{{\bf P_c} \phi^2}(\xi)
  %\qquad \chi_C(\xi):=\varphi_{\leq -C}(\jxi - 2 \lambda)
  %\\
  = - %\chi_C(\xi) 
  \int_0^t \exp\Big(is(2\lambda-\jxi)
  + i\frac{2c_2 \lambda}{\Gamma} \log \Big( 1+ \frac{\Gamma}{\lambda} Y_0^2 s\Big) \Big) 
  y(s)^2 ds \, e^{2i\Psi_\infty} \wt{{\bf P_c} \phi^2}(\xi)
  \\
  + g_\infty(\xi) + O_{L^\infty_\xi}(\e^\delta) \jt^{-\delta},
\end{split}
\end{align}
for some $g_{\infty} \in \varepsilon^{\delta} L^{\infty}_{\xi} .$
Therefore, assuming that there exist $\mathcal{N}(\wt{f})$  as in Lemma \ref{Nexist}
%\in \varepsilon^{\delta} L^{\infty}_t  L^{\infty}_{\xi}$
%and $h_\infty \in \e^\delta L^\infty_\xi$ such that 
%\cfp{Fix $\mathcal{N}(\wt{h})(t)$ to $f$}
%\begin{align} \label{hinfty}
%{\big\| \wt{h}(t) - \fp{\mathcal{N}(\wt{f})(t)} - h_\infty \big\|}_{L^\infty_\xi} \lesssim \e^\delta \jt^{-\delta}
%\end{align}
%for some $\delta >0$, 
the result \eqref{fasy}
follows by letting $F_\infty(t) := \mathcal{N}(\wt{f})(t)
  +  h_\infty - g_\infty \in \e^\delta L^{\infty}_t L^\infty_\xi$ since $f=h-g.$

To verify the last statement of Proposition \ref{main}, %it suffices to show
we look at the integral in \eqref{fasy}: %is $O(1)$: 
when $|\jxi - 2\lambda| \leqslant \delta_0 t^{-1}$ %for some small $\delta_0>0$,  
and $t \gtrsim Y_0^{-2} \approx \e_0^{-2}$ we have
\begin{align*}%\label{fasy1}
\begin{split}
& \int_0^t \exp \Big(is(\jxi-2\lambda)
  + i 2c_2\tfrac{\lambda}{\Gamma} \log\big(1 + \tfrac{\Gamma}{\lambda} Y_0^2 s\big) \Big) y(s)^2 ds \, e^{2 i \Psi_\infty} 
  \, \wt{{\bf P}_c\phi^2}(\xi) \\
  &= \frac{1}{2ic_2} \int_0^t e^{is (\jxi-2\lambda)} \partial_s \bigg(e^{2ic_2 \frac{\lambda}{\Gamma} \log \big(1 + \frac{\Gamma}{\lambda} Y_0^2 s \big)} \bigg) \, ds  \, e^{i 2\Psi_\infty} 
  \, \wt{{\bf P}_c\phi^2}(\xi)  \\ 
  & = \frac{\, e^{i 2\Psi_\infty} 
  \, \wt{{\bf P}_c\phi^2}(\xi)}{2ic_2} \Bigg( 
  \Big[e^{it(\jxi-2\lambda)} e^{2ic_2 \frac{\lambda}{\Gamma} \log \big(1 + \frac{\Gamma}{\lambda} Y_0^2 t \big)}
   - 1 \Big] - i\big(\jxi - 2 \lambda \big) \int_0^t e^{is (\jxi-2\lambda)} e^{2ic_2 \frac{\lambda}{\Gamma} \log \big(1 + \frac{\Gamma}{\lambda} Y_0^2 s \big)} ds \Bigg)
  \\
  & = \frac{1}{2ic_2} \Big[e^{2ic_2 \frac{\lambda}{\Gamma} \log \big(1 + \frac{\Gamma}{\lambda} Y_0^2 t \big)} 
  - 1 \Big] \, e^{i 2\Psi_\infty} 
  \, \wt{{\bf P}_c\phi^2}(\xi) + O(\delta_0).
\end{split}  
\end{align*}
%The quantity above has modulus $O(1)$ for $t \gtrsim Y_0^{-2} \approx \e_0^{-2}$ as claimed.
\end{proof}

\subsection{Estimates from \cite{LP}}
Before embarking on the proof, we collect some facts from \cite{LP} that will be used below.
We will use notation from \S1.6 of \cite{LP}, 
in particular concerning cutoffs and corresponding Littlewood-Paley projections,
such as projection to (distorted) frequencies of size $\approx 2^k$, $k\in \Z$, i.e.,
$P_k f = \wtF^{-1} (\varphi_k \wt{f})$, or size $\lesssim 2^k$, i.e. 
$P_{\leqslant k} f = \wtF^{-1} (\varphi_{\leqslant k} \wt{f})$.
We begin with some bounds satisfied by $h$ and $g$:

\begin{lemma}
%Let $Z = Z_\beta(\e,m) := \big(2^m \rho(2^m)\big)^{1-\beta} \big(2^m \varepsilon \big)^{\beta},$ 
%We write $f = h + g,$
%where $h$ and $g$ (localized at frequencies $\vert \xi \vert \approx \sqrt{4 \lambda^2 -1}$) 
%The functions $h$ and $g$ defined above satisfy the following properties: 
For all $t \geqslant \varepsilon^{-1}, t \approx 2^m,$ 
\begin{subequations}
\label{basic-bounds}
\begin{align} 
\label{basic-bounds-wh} \Vert \partial_{\xi} \wt{h}(t) \Vert_{L^2} & \lesssim Z, 
\\ 
\label{basic-bounds-eh} \Vert G(t) \Vert_{H^N} & \lesssim \varepsilon^{1-},  \ \ G \in \lbrace  g,h\rbrace  ,
\\ 
\label{basic-bounds-decay} \Vert e^{itL} G(t) \Vert_{L^q} 
  & \lesssim 2^{\big(3 \big(\frac{1}{q}-\frac{1}{2}\big) + 10 \delta_N\big)m} Z,  
\ \ G \in \lbrace  g,h\rbrace, \ \ q \in [2,6],  
\\ 
\label{basic-bounds-hardy} \Vert \varphi_{k_1} \wt{G}(t) \Vert_{L^1_{\xi}} & \lesssim 2^{(5/2-)k_1} \cdot Z,
  \ \ G \in \lbrace  g,h\rbrace,
\end{align}
\end{subequations}
where $\delta_N := 5/N.$

\noindent
For the component $g$ we also have, for $t \approx 2^m,$ 
\begin{subequations} \label{bounds-g}
\begin{align} 
 \label{bounds-ginfty} 
 {\Vert  \wt{g}(t,\xi) \Vert}_{L^{\infty}_{\xi}} & \lesssim 2^m \rho(2^m) m, 
 \\ 
 \label{bounds-g-derL1} {\Vert \partial_{\xi} \wt{g}(t,\xi) \Vert}_{L^1_{\xi}} & \lesssim 2^m \rho(2^m) m^2, 
 \\ 
 \label{bounds-g-derL2} \Vert \partial_{\xi} \wt{g}(t,\xi) \Vert_{L^2} & \lesssim 2^{3m/2} \rho(2^m) m^2.
\end{align}
\end{subequations}
Finally we have estimates for the time-derivative of the two components:
with $M_0 := \delta_N m$
\begin{subequations} \label{time-der}
\begin{align} 
\label{time-der-g} \Vert e^{itL} \partial_t g(t) \Vert_{L^p} & \lesssim \rho(t),
  \quad p \geq 1,
\\ 
\label{time-der-h} \Vert e^{itL} \partial_t h(t) \Vert_{L^q} & \lesssim 2^{-m + 30 M_0} Z^2,
  \quad q > (3/2)-.
\end{align}
\end{subequations}
\end{lemma}

\begin{proof}
\eqref{basic-bounds-wh} is one of the bootstrap bounds established on $h$
as part of the proof of Theorem \ref{maintheo}; see $(1.50)$ in \cite{LP}.
\eqref{basic-bounds-eh} is a consequence of the second a priori assumption in $(1.50)$ of \cite{LP}
and a direct estimation of the $H^N$ norm of $g,$ see for example $(1.57)$ in \cite{LP}. 
\eqref{basic-bounds-decay} is a consequence of dispersive estimates for the free 
Klein-Gordon semi-group and our bootstrap estimates, %a priori assumptions, 
see Lemma $4.5$ in \cite{LP}. 
\eqref{basic-bounds-hardy} is a consequence of Bernstein's and Hardy's inequality, 
the bounds %a priori assumption 
on $h$ and the $L^{\infty}_{\xi}$ bound on $g,$ see \eqref{bounds-ginfty}. 

\eqref{bounds-g} and \eqref{time-der-g} are directly borrowed from Lemmas 4.7 and 4.6 in \cite{LP}.
Only \eqref{time-der-h} requires a proof. It follows from the identity
\begin{align*}
\partial_s h(s) = e^{-isL} \big(L^{-1} \Im w \big)^2 + e^{-isL} O(\rho(s)) \theta 
  + e^{-isL} \big(a(s) \phi \cdot L^{-1} \Im w \big),
\end{align*}
and we can conclude using \eqref{basic-bounds-decay} and interpolation with \eqref{basic-bounds-eh}.
\end{proof}

\smallskip
\section{Preliminaries} 
We start by recording a consequence of \eqref{gasy}
which improves \eqref{bounds-ginfty}:

\begin{lemma} \label{gLinftyO(1)}
We have $\Vert \wt{g}(t) \Vert_{L^{\infty}_{\xi}} \lesssim 1.$
\end{lemma}

\begin{proof}
It suffices to show that the time integral in \eqref{gasy} is $O(1)$.
We first assume $|\jxi -2\lambda |^{-1} \leqslant t$ and split the time integral into
an integral from $0$ to $|\jxi -2\lambda |^{-1}$ and its complement.
For the first piece, recalling the definition of $y(s) = Y_0 \, (1 + (\Gamma/\lambda) s Y^2_0)^{-1/2}$, 
we write
\begin{align*}
&\bigg \vert \int_0^{|\jxi -2\lambda |^{-1}} 
  \exp \bigg( is (\jxi-2\lambda) + i \frac{2 c_2 \lambda}{\Gamma} 
  \log \bigg(1 + \frac{\Gamma}{\lambda} Y_0^2 s \bigg)  \bigg) y^2(s) \,ds  \bigg \vert 
  \\
& = \frac{1}{2 c_2} \bigg \vert \int_0^{|\jxi -2\lambda |^{-1}} \exp \big( is (\jxi-2\lambda) \big) 
  \partial_s \exp \bigg(i \frac{2 c_2 \lambda}{\Gamma} 
  \log \bigg(1 + \frac{\Gamma}{\lambda} Y_0^2 s \bigg) \bigg) \, ds  \bigg \vert  
  \lesssim  1,
\end{align*}
where the last inequality can be seen integrating by parts in $s$. 
For the other piece we instead write
\begin{align*}
&\bigg \vert \int_{|\jxi -2\lambda |^{-1}}^t  
  \exp \bigg( is (\jxi-2\lambda) 
  + i \frac{2 c_2 \lambda}{\Gamma} \log \bigg(1 + \frac{\Gamma}{\lambda} Y_0^2 s \bigg)  \bigg) y^2(s) \, ds  \bigg \vert 
  \\
& = \vert \jxi - 2 \lambda \vert^{-1} \bigg \vert \int_{|\jxi -2\lambda |^{-1}}^t  
  \Big( \partial_s \exp \big( is (\jxi-2\lambda) \big) \Big) \, \exp \bigg( i \frac{2 c_2 \lambda}{\Gamma} 
  \log \bigg(1 + \frac{\Gamma}{\lambda} Y_0^2 s \bigg) \bigg)y^2(s) \,ds  \bigg \vert.
\end{align*}
Integrating by parts gives the boundary term
\begin{align*}
\vert \jxi - 2 \lambda \vert^{-1} \bigg \vert 
  \bigg( \exp \big( is (\jxi-2\lambda) \big) \bigg)  \exp \bigg( i \frac{2 c_2 \lambda}{\Gamma} 
  \log \bigg(1 + \frac{\Gamma}{\lambda} Y_0^2 s \bigg) \bigg)y^2(s) \,ds  \bigg \vert_{s=|\jxi -2\lambda |^{-1}}^{s=t}
  \bigg \vert \lesssim 1,
  %\\
  %\lesssim \vert \jxi - 2 \lambda \vert^{-1}
  %Y_0^2 \, (1 + (\Gamma/\lambda) |\jxi -2\lambda |^{-1} Y^2_0)^{-1}
\end{align*}
and a bulk term which is also $O(1)$ since $|(d/ds) y^2(s)| \lesssim y^4(s)$.
%where we used that for $s \approx 2^m,$ we have $\vert (d/ds)^{a} y^2(s) \vert \lesssim 2^{-(1+a)m}. $ 
%This yields the desired result.
Finally, if $|\jxi -2\lambda |^{-1} \geqslant t$ one can estimate as we did for the first integral above.
\end{proof}

The control of the $L^{\infty}_{\xi}$ norm 
and the proof of the main scattering statement for $\wt{h}$, see \eqref{hinfty}, 
is based on a decomposition obtained from Duhamel's formula. 
This is done in two steps below, Lemma \ref{lemdecomph} %, see \eqref{decomph}, 
and Lemma \ref{lemhdecomp} %, see \eqref{decompFL} with 
(see also \eqref{FHL}).

\smallskip
We introduce a partition of unity of the interval $[0,t]$ to decompose our time integrals. %, $t\leq T$.
Let $\tau_0,\tau_1,\cdots, \tau_{L+1}: \R \to [0,1]$ denote cut-off functions, 
where the integer $L$ is chosen so that $|L-\log_2 (t+2)| < 2$, with the following properties:
\begin{align}\label{timedecomp}
\begin{split}
& \sum_{n=0}^{L+1}\tau_n(s) = \mathbf{1}_{[0,t]}(s),  
\qquad \supp (\tau_0) \subset [0,2], \quad  \supp (\tau_{L+1}) \subset[t/4,t],
\\
& \mbox{and} \quad \supp(\tau_n) \subseteq [2^{n-1},2^{n+1}], 
  \quad |\tau_n'(t)|\lesssim 2^{-n}, \quad \mbox{for} \quad n= 1,\dots, L.
\end{split}
\end{align}

Recall the definitions \eqref{defg} and \eqref{defA}-\eqref{A-B}.
Lemma 4.1 in \cite{LP} gives us the following:

\begin{lemma} \label{lemdecomph}
We have
\begin{align} \label{decomph}
\wt{h} = \wt{f_0} + F %F_L 
  + \sum_{m=0}^{L+1} (S_m + M_m), %+ F_{H,m},
\end{align}
where the `source terms' are (recall the definition of $\chi_C$ below \eqref{defg})
\begin{subequations}
\label{S}
\begin{align}
\nonumber
S_m & = S_{1,m} + S_{2,m} + \lbrace \mbox{similar and better terms} \rbrace,
\\
\label{Scubic}
& S_{1,m} := %\chi_C(\xi) 
  \int_0 ^t e^{-is(\jxi-2\lambda)} (A^2(s)-B^2 (s)) \, \mf{\theta}(\xi) \, \tau_m(s) ds,
\\
\label{S1-chi}
& S_{2,m} := (1-\chi_C(\xi))   
  \int_0 ^t e^{-is(\jxi-2\lambda)} B^2 (s) \, \mf{\theta}(\xi) \, \tau_m(s) ds,   
  \end{align}
\end{subequations}
the `mixed terms' are
\begin{align}
\notag M_m & = -\frac{i}{2} M_{1,m}  + \lbrace \mbox{similar and better terms} \rbrace,
\\
\label{M1}
M_{1,m} & := \int_0 ^t B(s) \int_{\R^3} e^{-is(\jxi - \jeta - \lambda)} 
  \jeta^{-1} \wt{f}(s,\eta) \nu(\xi,\eta) d\eta \,\tau_m (s) ds , 
  \\
\nonumber
\nu(\xi,\eta) & := \frac{1}{(2 \pi)^3} \int_{\R^3} \overline{\psi(x,\xi)} \psi(x,\eta) \phi(x) dx,
\end{align}
and the `continuous/field' self interactions are
\begin{align}
\label{Fepssquad}
\begin{split}
 F & := -\frac{1}{4}\sum_{\epss \in \lbrace +,- \rbrace} \eps_1\eps_2 F_{\epss,m}\big(f_{\eps_1},f_{\eps_2}\big), \qquad f_+=f, \,\,\, f_- = \bar{f},
\\
F_{\epss}(a,b) &:= \int_0^t \int_{\R^6} \frac{e^{-is \Phi_\epss (\xi,\eta,\sigma)}}{\jeta \jsig} 
    \wt{a}(s,\eta) \wt{b}(s,\sigma) \mu(\xi,\eta,\sigma) d\eta d\sigma \, ds, \\
\mu(\xi,\eta,\sigma) &  := \frac{1}{(2 \pi)^{9/2}} \int_{\R^3} \overline{\psi(x,\xi)} \psi(x,\eta) \psi(x,\sigma) dx,    
\end{split}
\end{align}
where the phases are defined by 
\begin{align}\label{phepss}
\Phi_{\epss}(\xi,\eta,\sigma) = \jxi -\eps_1\jeta -\eps_2\jsig.
\end{align}

\end{lemma}

We split the $F$ terms from the above lemma
into low and high frequencies by letting %$F =F_H +  F_L,$ where 
\begin{align}\label{FHL}
\begin{split}
%-4 
F :=  F_{L} + F_H, %\sum_{m} F_{H,m}}
\qquad 
F_{L}(t,\xi) %:= \sum_{\epss} \eps_1\eps_2 F_{\epss}\big(f_{\eps_1},f_{\eps_2}\big),
  := \int_0^t \varphi_{\leqslant M_0(s)}(\xi) \, e^{-is\jxi} \wtF \Big( \frac{P_{\leqslant M_0(s)} \Im w}{L}\Big)^2 \,  ds,
\\
M_0(s) := \delta_N \log \langle s \rangle, \quad \delta_N := 5/N.
\end{split}
\end{align}
We will also let $F_{H,m}$ denote the term $F_H$ where the integral in time is localized to $s\approx 2^m$
using the cutoffs $\tau_m$ (as in \eqref{M1} for example).

We need to further decompose $F_L$ using some more refined information on the measure $\mu$
contained in Proposition 5.6 of \cite{LP},
which we reproduce below.

\begin{lemma} \label{mudecomp}
Let $M_0$ be a fixed parameter, and let\footnote{The parameter $N_2$
can be fixed large enough as a fraction of $N$.} $N_2$ be a sufficiently large integer.
%cfp{Take care of $N_2$\dots Maybe fix at the beginning?
%Then say remainder estimates are for $N_2/C$ many derivatives}
We can decompose $\mu$ into a singular part (apex $S$), a regular part (apex $R$) 
and a remainder part (index $Re$) 
\begin{align}\label{mudecomp00}
\mu = \mu^S + \mu^R + \mu^{Re}.
\end{align}
These three components satisfy the following structural properties (up to irrelevant constants):
\begin{align}\label{mudecomp0}
\begin{split}
& \mu^S(\xi,\eta,\sigma) := \delta_0(\xi-\eta-\sigma) + \mu_1^S(\xi,\eta,\sigma) 
  + \mu_2^S(\xi,\eta,\sigma) + \mu_3^S(\xi,\eta,\sigma),
\\
& \mu^R(\xi,\eta,\sigma) :=  \mu_1^R(\xi,\eta,\sigma), 
%\\ 
%& \mu^{Re}(\xi,\eta,\sigma) := \mu_2^R(\xi,\eta,\sigma)+ \mu_3^R(\xi,\eta,\sigma) + \mu^{Re'}(\xi,\eta,\sigma), 
\end{split}
\end{align}
where the right-hand sides are defined as follows:

\smallskip
\setlength{\leftmargini}{1.5em}
\begin{itemize}

\smallskip
\item $\mu_1^S$ and $\mu_1^R$ are given by the formulas
%The singular and regular parts are further decomposed as
%\begin{align}
%& \mu^S(\xi,\eta,\sigma) := \delta_0(\xi-\eta-\sigma) + \mu_1^S(\xi,\eta,\sigma) 
%  + \mu_2^S(\xi,\eta,\sigma) + \mu_3^S(\xi,\eta,\sigma), 
%  \\
%  & \mu^R(\xi,\eta,\sigma) :=  \mu_1^R(\xi,\eta,\sigma),
%\end{align}
%where
\begin{align}\label{mu1SR}
\mu_1^{\ast}(\xi,\eta,\sigma) 
  & := \nu_1^{\ast} (-\xi +\eta,\s) +  \nu_1^{\ast} (-\xi +\s,\eta) +  \overline{\nu_1^{\ast} (-\eta-\s,\xi)}, 
  \qquad \ast \in \lbrace S,R \rbrace
\end{align}
with
%are defined in terms of the decomposition obtained in \cite{PS},
%see Proposition \ref{decomp-meas-PS} where the results are summarized. 
%First let us define (all quantities in the right-hand side are defined in Proposition \ref{decomp-meas-PS})
\begin{align}\label{nu1S}
\begin{split}
\nu_1^S(p,q) & := \varphi_{\leqslant -M_0-5}(\vert p \vert - \vert q \vert) \Big[ 
  \nu_0 (p,q)
  \\ & + \frac{1}{\vert p \vert} \sum_{a=1}^{N_2} \sum_{J \in \Z} %\geqslant M_0+5} 
  b_{a,J}(p,q) \cdot 2^J K_a \big(2^J (\vert p \vert - \vert q \vert) \big) \Big]
\end{split}
\end{align}
where $K_a$ are Schwartz functions, $b_{a,J}$ are symbols satisfying the bounds 
\begin{align}
\label{Propnu+2.1}
\sum_{J\in\Z} 
  \big|  \varphi_P(p) \varphi_Q(q) \nabla_p^\alpha \nabla_q^\beta  b_{a,J}(p,q) \big| \lesssim 
  2^{-|\alpha|P} \big(2^{|\alpha |Q}  + 2^{(1-|\beta|)Q_-}\big) \cdot \mathbf{1}_{\{|P-Q|< 5\}},
\end{align}
for all $P,Q \leq M_0$, $|\alpha|+|\beta| \leqslant N_2,$ and\footnote{We are using 
the notation $b_0(p/|p|,q)$ here which slightly differs from the one %in formulas (5.11)-(5.12)
in Proposition 5.1 of \cite{PS} (or Proposition 5.6 in \cite{LP}); 
in this latter the symbol appearing in the analogous formulas (5.11)-(5.12) was denoted $b_0(p,q)$ instead.
The symbol in \eqref{nu0} actually denotes (up to a constant) the symbol 
$g(-p/|p|,q)$ which appears in Lemma 5.2 of \cite{PS}.
The relation between the two symbols in \cite{PS} is (up to a constant) $b_0(p,q) = g(-p/|p|,q)$,
consistently with \eqref{nu0}-\eqref{estimatesg}.
%not exactly the same symbol that appears in (5.11)-(5.12)
%We refer to (5.20) and (5.21) in the statement of Lemma 5.2 in \cite{PS}
} 
%\cfp{$b_0$ to check} % or adjust in Lemma \ref{lembil} for consistency}
\begin{align}\label{nu0}
\nu_0(p,q) :=  \frac{b_0(p/|p|,q)}{|p|} \Big[ i\pi \, \delta(|p|-|q|) + \pv \frac{1}{|p|-|q|} \Big]
\end{align}
with
\begin{align} \label{estimatesg}
\begin{split}
%b_0(p,q) &= 2 i \pi g(-p/\vert p \vert,q),  \\
\big \vert \varphi_{P}(p) \varphi_{Q}(q) 
  \nabla_{q}^{\beta} %g(p/\vert p \vert ,q) 
  b_0(p/|p|,q) \big \vert &\lesssim  2^{(1-\vert \beta \vert)Q_{-}} \cdot \textbf{1}_{\lbrace \vert P-Q \vert <5 \rbrace}, 
  \  \ Q_{-} := \min(Q,0), \ \ 1 \leqslant %\vert \alpha \vert + 
  \vert \beta \vert \leqslant N_2.
\end{split}
\end{align}

\smallskip
\item $\mu_2^S$ %and $\mu_2^R$ 
is given by the formulas
\begin{align}\label{mu2SR}
\mu_2^{S}(\xi,\eta,\sigma) 
  & := \nu_{2}^{S,1} (\xi,\eta,\s) + \nu_{2}^{S,2} (\xi,\eta,\s) + \nu_{2}^{S,2} (\xi,\s,\eta), 
  %\qquad \ast \in \lbrace S,R \rbrace.
\end{align}
with
\begin{align}\label{nu2S1}
\begin{split}
\nu_{2}^{S,1}(\xi,\eta,\s) := \frac{1}{|\xi|} \sum_{i=1}^{N_2} 
 \sum_{J \in \Z %M_0+5
 } \varphi_{\leqslant -M_0-5} (|\xi|-|\eta|-|\s|)
 \\
 \times b_{i,J}(\xi,\eta,\s) \cdot K_i\big( 2^J (|\xi|-|\eta|-|\s|) \big),  
\end{split}
\end{align}
and %(all quantities in the right-hand side are defined in Proposition \ref{decomp-meas-PS})
\begin{align}\label{nu2S2}
\begin{split}
\nu_{2}^{S,2}(\xi,\eta,\s) = \frac{1}{\vert \eta \vert} \sum_{\epsilon \in \lbrace 1, -1 \rbrace} 
  \sum_{i=1}^{N_2} \sum_{J \in \Z %\geqslant M_0+5
  } \varphi_{\leqslant -M_0-5} (|\xi|+\epsilon|\eta|-|\s|) 
  \\ 
  \times b_{i,J}^{\epsilon}(\xi,\eta,\s) \cdot 
  K_i \big(2^J (|\xi|+\epsilon|\eta|-|\s|) \big), 
\end{split}
\end{align}
where $K_i$ are Schwartz functions and $b_{i,J}$ and $b_{i,J}^\epsilon$  are symbols satisfying the bounds 
\begin{align}\label{Propnu212}
\begin{split}
  & \big| \varphi_k(\xi) \varphi_{k_1}(\eta) \varphi_{k_2}(\sigma) 
  \nabla_\xi^a \nabla_{\eta}^\alpha \nabla_\sigma^\beta  b_{i,J}(\xi,\eta,\sigma) \big| 
  \\ 
  & \lesssim 2^{-|a|k} \cdot \big( 2^{|a|\max(k_1,k_2)} + 2^{(1-|\alpha|)k_1} 2^{(1-|\beta|)k_2} \big)
  \mathbf{1}_{\{ |k-\max(k_1,k_2)|<5 \}},
\end{split}
\end{align}
and 
\begin{align}\label{Propnu222}  
\begin{split}
%\sum_{J\in\Z}  %% Don't need sum
  & \big| \varphi_k(\xi) \varphi_{k_1}(\eta) \varphi_{k_2}(\sigma) 
  \nabla_{\xi}^a \nabla_\eta^\alpha \nabla_{\sigma}^\beta  b_{i,J}^\pm(\xi,\eta,\sigma) \big| 
  \\ 
  & \lesssim 2^{-|\alpha|k_1} \cdot \big( 2^{|\alpha|\max(k,k_2)} + 2^{(1-|a|)k} 2^{(1-|\beta|)k_2} \big)
  \mathbf{1}_{\{\max(k,k_1,k_2)-\mbox{\tiny$\mathrm{med}$}(k,k_1,k_2)<5\}},
\end{split}
\end{align}
for all $k,k_1,k_2 \leq M_0$, and $|a| + |\alpha|+|\beta| \leq N_2,$ respectively.

\smallskip
\item $\mu_3^S$ %and $\mu_3^R$ 
is given by the formula
\begin{align}\label{mu3S}
\begin{split}
\mu_{3}^S (\xi,\eta,\s) := \sum_{i=0}^{N_2} \sum_{J \in \Z %\geqslant M_0+5
  } 
  \varphi_{\leqslant -M_0-5}(|\xi|-|\eta|-|\s|) 
  \\ \times b_{i,J}(\xi,\eta,\s) \cdot K_i \big(2^J(|\xi|-|\eta|-|\s|) \big), 
\end{split}
\end{align}
where $K_i$ are Schwartz functions
and $b_{i,J}$ are symbols satisfying the bounds stated in \eqref{Propnu+2.1}.
\end{itemize}
\end{lemma}

We list some pointwise bounds satisfied by the regular measures $\nu, \nu_1^R$ and $\mu^{Re}$,
which can again be found in \cite{LP} (see Proposition 5.6 and estimate (8.6)):

\begin{lemma}
Let $M_0$ and $N_2$ be as in Lemma \ref{mudecomp}.
For the measure $\nu(\xi,\eta)$ we have
\begin{align} \label{nuest}
 \big| \varphi_{k}(\xi) \varphi_{k_1}(\eta) \nabla_\xi^a \nabla^b_\eta \nu(\xi,\eta) \big| 
  \lesssim 1, \qquad |a|,|b|\geqslant 1.
\end{align}
For $\mu^{Re}$ we have, for all $|a| + |\alpha| + |b| \leqslant N_2$,
%(where $N_2$ can be fixed large enough as a fraction of $N$),
\begin{align}\label{muReest}
\vert \varphi_k(\xi) \varphi_{k_1}(\eta) \varphi_{k_2}(\sigma) \nabla_{\xi}^a \nabla_{\eta}^{\alpha} \nabla_{\s}^{\beta} \mu^{Re} (\xi,\eta,\s) \vert 
  & \lesssim 2^{-2 \max \lbrace k_1,k_2,k_3 \rbrace} 
  \cdot 2^{-(\vert \alpha_{a} \vert + \vert \alpha_{e} \vert) \max \lbrace k_1, k_2,k_3 \rbrace}  
  \\ \notag & \cdot \max \big(1 , 2^{(1-\vert \alpha_{i} \vert) \min \lbrace k_1, k_2 ,k_3 \rbrace} \big) 
  \cdot 2^{5M_0(|a| + |\alpha| + |\beta| +1)},
\end{align}
where $\alpha_{i}$ denotes the order of differentiation of the variable that is located at the smallest frequency, 
$\alpha_{e}$ is similarly defined for the middle frequency and $\alpha_{a}$ for the largest one.

Finally for $\nu_1^R,$ we have for all $|a|+|b| \leq N_2,$
\begin{align} \label{nuRest}
\begin{split}
& \vert \varphi_{k_1}(\eta) \varphi_{k_2}(\sigma) \nabla_{p}^a \nabla_{q}^b \nu_1^R (\eta,\sigma) \vert
\\
& \lesssim \begin{cases}
2^{-2 k_1} \, 2^{-(|a|+ |b|)k_1} \cdot 2^{(2 + |a| + |b|)5M_0} & \textrm{if}  %A \mapsto M_0
  \quad |k_1-k_2| < 5 
 \\
2^{-2 (k_1 \vee k_2)} \cdot 2^{-|a| (k_1 \vee k_2)} \max \lbrace 1 , 2^{(1-\vert b \vert) k_2^-}  \rbrace
  \cdot 2^{(|a| + |b| + 2)5M_0} &  \textrm{if} \quad |k_1-k_2|  \geqslant 5.
\end{cases}
\end{split}
\end{align}
\end{lemma}

We now define bilinear operators associated with the measures above.
\begin{definition}
For a general symbol $b=b(\xi,\eta,\s)$ let us denote
\begin{subequations}\label{idop}
\begin{align}
\label{idop0}
T_0[b](G,H)(x) & := \whF^{-1}_{\xi\rightarrow x} \iint_{\R^3\times\R^3} G(\eta) H(\s) 
  \,b(\xi,\eta,\s) \, \delta_0(\xi-\eta-\s) \, d\eta d\s,
\\
\label{idop11}
T_1^{\ast,1}[b](G,H)(x) & := \whF^{-1}_{\xi\rightarrow x} \iint_{\R^3\times\R^3} G(\xi-\eta) H(\sigma) 
  \,b(\xi,\eta,\sigma)\, \nu_1^{\ast}(\eta,\sigma) \, d\eta d\sigma, \quad \ast \in\{S,R\},
\\
\label{idop12}
T_1^{\ast,2}[b](G,H)(x) &:= \whF^{-1}_{\xi\rightarrow x} \iint_{\R^3\times\R^3} G(-\eta-\sigma) H(\sigma) 
  \,b(\xi,\eta,\sigma)\, \overline{\nu_1^{\ast}(\eta,\xi)} \, d\eta d\sigma, \quad \ast \in\{S,R\},
\\
\label{idop2}
T_2^{S} [b](G,H)(x) & := \whF^{-1}_{\xi\rightarrow x} \iint_{\R^3\times\R^3} G(\eta) H(\sigma) 
  \,b(\xi,\eta,\sigma)\, \mu_2^{S}(\xi,\eta,\sigma) \, d\eta d\sigma,
\\
\label{idop3}
 T_3^{S} [b](G,H)(x) & := \whF^{-1}_{\xi\rightarrow x} \iint_{\R^3\times\R^3} G(\eta) H(\sigma) 
  \,b(\xi,\eta,\sigma)\, \mu_3^{S}(\xi,\eta,\sigma) \, d\eta d\sigma,
\\
\label{theomuRe}
T^{Re}[b](G,H)(x) & := \whF^{-1}_{\xi\rightarrow x} \iint_{\R^3\times\R^3} G(\eta) H(\sigma) 
  \,b(\xi,\eta,\sigma)\, \mu^{Re}(\xi,\eta,\sigma) \, d\eta d\sigma.
\end{align}
\end{subequations}
\end{definition}

We also adopt the following notation:
%%\begin{remark}[Notation for symbols]
for a general symbol $\underline{b}$ we let (omitting the dependence on $k_1,k_2\in\Z$)
\begin{align}\label{symnot}
\begin{split}
& b(\xi,\eta,\sigma) := \underline{b}(\xi,\eta,\sigma) \varphi_{k_1}(\eta) \varphi_{k_2}(\sigma), 
\qquad b_1(\xi,\eta,\sigma) := \underline{b}(\xi,\xi-\eta,\sigma) \varphi_{k_1}(\eta) \varphi_{k_2}(\sigma),
\\
& b_1'(\xi,\eta,\sigma) := \underline{b}(\xi,\sigma,\xi-\eta)\varphi_{k_1}(\eta) \varphi_{k_2}(\sigma), 
\qquad b_2(\xi,\eta,\sigma) := \underline{b}(\xi,-\eta-\sigma,\sigma)\varphi_{k_1}(\eta) \varphi_{k_2}(\sigma).
\end{split}
\end{align}

With the building blocks \eqref{idop} we can further decompose $F_L$ 
into `singular' and `regular' parts.
The singular parts (denoted with an apex $S$) are defined so that the phases \eqref{phepss} are lower bounded on 
their support, and we can apply a normal form transformation;
this will give rise to boundary terms that are quadratic (see \eqref{Sbdry} below)
and cubic bulk terms (see \eqref{Sint} below).
The regular parts (see \eqref{Rint} and \eqref{Reint} below) 
cannot be transformed but will satisfy better bilinear estimates.
This singular vs. regular decomposition is the content of the next lemma, and we refer the reader to \S5.4 of \cite{LP} 
for more details.

\begin{lemma}\label{lemhdecomp}
We have (up to irrelevant constants)
\begin{align} \label{decompFL}
F_{L}(t) = \sum_{\epss \in \lbrace +,- \rbrace} F^{S,(1)}_{\epss}(t) 
  - F^{S,(1)}_{\epss}(0) +\sum_{m=0}^{L+1} F^{S,(2)}_{\epss,m}(t) + F^R_{\epss,m}(t) + F^{Re}_{\epss,m}(t)
\end{align}
with 
\begin{align} \label{Sbdry}
\begin{split}
F^{S,(1)}_{\epss}(t) &= e^{-it \jxi} \varphi_{\leqslant M_0(t)} (\xi)
    \sum_{k_1,k_2 \leqslant M_0(t) %\langle t \rangle^{\delta_N}
    } \sum_{G,H \in \lbrace g,h \rbrace} \widehat{\mathcal{F}}_{x \rightarrow \xi} 
    \Bigg[T_0[b %_0
    ]\big( e^{\eps_1 it\jxi}\wt{G_{\eps_1}}, 
    e^{\eps_2 it\jxi}\wt{H_{\eps_2}} \big) 
  \\
  & +T_1^{S,1}[b_1]\big( e^{\eps_1 it\jxi}\wt{G_{\eps_1}}, 
    e^{\eps_2 it\jxi}\wt{H_{\eps_2}} \big) + T_1^{S,1}[b_1']\big( e^{\eps_2 it\jxi}\wt{H_{\eps_2}},  e^{\eps_1 it\jxi}\wt{G_{\eps_1}} \big) 
\\   
    &+ T_1^{S,2}[b_2]\big( e^{\eps_1 it\jxi} \wt{G_{\eps_1}}, 
    e^{\eps_2 it\jxi}\wt{H_{\eps_2}} \big)  + \sum_{i=2,3} T_i^S[b]\big( e^{\eps_1 it\jxi}\wt{G_{\eps_1}}, 
    e^{\eps_2 it\jxi}\wt{H_{\eps_2}} \big) \Bigg] ,
\end{split}
\end{align}
and (disregarding analogous symmetric terms)
\begin{align} \label{Sint}
\begin{split}
F^{S,(2)}_{\epss,m}(t) &= \int_0^t e^{-i s \jxi} \sum_{k_1,k_2 \leqslant M_0(s)} 
  \sum_{G,H \in \lbrace g,h \rbrace} \widehat{\mathcal{F}}_{x \rightarrow \xi} 
  \Bigg[T_0[b %_0
  ]\big( e^{\eps_1 is \jxi} \partial_s \wt{G_{\eps_1}}, 
    e^{\eps_2 is \jxi}\wt{H_{\eps_2}} \big) 
  \\
  & +T_1^{S,1}[b_1]\big( e^{\eps_1 is \jxi} \partial_s \wt{G_{\eps_1}}, 
    e^{\eps_2 is \jxi}\wt{H_{\eps_2}} \big) + T_1^{S,1}[b_1']
    \big( e^{\eps_2 is \jxi}\wt{H_{\eps_2}},  e^{\eps_1 is \jxi} \partial_s \wt{G_{\eps_1}} \big) 
\\   
    &+ T_1^{S,2}[b_2]\big( e^{\eps_1 is \jxi} \partial_s \wt{G_{\eps_1}}, 
    e^{\eps_2 is \jxi}\wt{H_{\eps_2}} \big)  + \sum_{i=2,3} T_i^S[b]\big( e^{\eps_1 is\jxi}
    \partial_s \wt{G_{\eps_1}}, 
    e^{\eps_2 is\jxi}\wt{H_{\eps_2}} \big) \Bigg] \, \tau_m(s) ds,
\end{split}
\end{align}
where, using the notation \eqref{symnot}, we define the symbols through
%$b$ denotes the symbol
\begin{align}\label{Sintb}
\underline{b}(\xi,\eta,\sigma) = \frac{\varphi_{\leqslant M_0}(\jxi)}{\jeta \jsigma \Phi_{\epss}(\xi,\eta,\sigma)}.
\end{align}
Moreover 
\begin{align} \label{Rint} 
\begin{split}
F^R_{\epss,m} & := F^{R,(1)}_{\epss,m} + F^{R,(2)}_{\epss,m}  
\\
F^{R,(1)}_{\epss,m}(t) & := \int_0^t e^{-i s \jxi} \sum_{k_1,k_2 \leqslant M_0(s) } 
    \sum_{G,H \in \lbrace g,h \rbrace} \widehat{\mathcal{F}}_{x \rightarrow \xi} 
    \Bigg[ T_1^{R,1}[m_1]\big( e^{\eps_1 is \jxi}\wt{G_{\eps_1}}, 
    e^{\eps_2 is \jxi}\wt{H_{\eps_2}} \big)\\
    & + T_1^{R,1}[m_1']\big( e^{\eps_2 is \jxi}\wt{H_{\eps_2}},  e^{\eps_1 is \jxi}\wt{G_{\eps_1}} \big)\Bigg] 
    \, \tau_m(s)ds,
    \\
F^{R,(2)}_{\epss,m}(t) & := \int_0^t e^{-i s \jxi} \sum_{k_1,k_2 \leqslant M_0(s)
  } \sum_{G,H \in \lbrace g,h \rbrace} 
  \widehat{\mathcal{F}}_{x \rightarrow \xi} \bigg[ T_1^{R,2}[m_2]\big( e^{\eps_1 is \jxi} \wt{G_{\eps_1}}, 
  e^{\eps_2 is \jxi}\wt{H_{\eps_2}} \big)\bigg] \, \tau_m(s) ds ,
\end{split}
\end{align}
with, using again the notation \eqref{symnot},
\begin{align}\label{Rintm}
\underline{m}(\xi,\eta,\sigma) := \varphi_{\leqslant M_0}(\jxi)%\varphi_{k_1}(\eta) \varphi_{k_2}(\sigma)
  \jeta^{-1} \jsigma^{-1}.
\end{align}
Finally
\begin{align}\label{Reint}
F^{Re}_{\epss,m}(t) & := \int_0^t \int_{\R^6} \frac{e^{-is \Phi_\epss (\xi,\eta,\sigma)}}{\jeta \jsig} 
    \wt{f_{\epsilon_1}}(s,\eta) \wt{f_{\epsilon_2}}(s,\sigma) \mu^{Re}(\xi,\eta,\sigma) d\eta d\sigma \, \tau_m (s) ds.
\end{align}
\end{lemma}

In what follows, we will often omit the indexes $\eps_1,\eps_2$ in the notation for the bilinear terms above,
since these should create no confusion (and they will often play no role).

Next we recall H\"{o}lder type bounds satisfied by the operators in \eqref{idop}:
\begin{theorem}\label{bilinearmeas}
Let a symbol $b=b(\xi,\eta,\s)$ be given such that
\begin{align*}%\label{theomu1asb1}
\textrm{supp}(b) \subseteq \big\{ (\xi,\eta,\sigma) \in \mathbb{R}^9\,:\, |\xi|+|\eta|+|\sigma| \leq 2^A
  %, \, |\eta| \approx 2^{k_1}, \, |\sigma| \approx 2^{k_2}
  \big\},
\end{align*}
for some $A\geq1$ and such that,
for $\vert \xi \vert \approx 2^K, \vert \eta \vert \approx 2^{L}, \vert \s \vert \approx 2^M$,
\begin{align}\label{bilmeasbest}
\big \vert \nabla_{\xi}^a \nabla_{\eta}^{\alpha} \nabla_{\s}^{\beta} 
  b(\xi,\eta,\s)  \big \vert \lesssim 2^{-\vert a \vert K - \vert \alpha \vert L - \vert \beta \vert M} 
  \cdot 2^{(\vert a \vert + \vert \alpha \vert + \vert \beta \vert)A},
  \qquad \vert a \vert , \vert \alpha \vert , \vert \beta \vert \leqslant 4.
\end{align}
Let $p,q,\in [1,\infty),$ and $r \geqslant 1$ with %\cfp{equality ?}
$$\frac{1}{p} + \frac{1}{q} = \frac{1}{r},$$
and assume there is $10A \leqslant D \leqslant 2^{A/10}$ such that 
\begin{align}\label{bilmeasD}
\mathcal{D}(G,H) := \Vert G \Vert_{L^2} \Vert H \Vert_{L^2} 
  + \min \big( \Vert \partial_{\xi} G \Vert_{L^2} \Vert H \Vert_{L^2},  
  \Vert G \Vert_{L^2} \Vert \partial_{\xi} H \Vert_{L^2} \big) \leqslant 2^D.
\end{align}
Then, there exists an absolute constant $C_0$ (for example, $C_0:=65$ is suitable), such that 
the following bilinear bounds hold for the operators defined in \eqref{idop}:
\begin{subequations} \label{mainbilin}
\begin{align}
\label{mainbilinS1} 
& \big \Vert P_K T_1^{S,i}[b](G,H) \big \Vert_{L^r} , 
  \lesssim  \Vert \widehat{G} \Vert_{L^p} \Vert \widehat{H} \Vert_{L^q} \cdot 2^{C_0 M_0} 
  + 2^{-D} \mathcal{D} (G,H),  \ \ i=1,2,
\\
& \label{mainbilinS23}
\big \Vert P_K T_i^{S}[b](G,H) \big \Vert_{L^r} 
  \lesssim \Vert \widehat{G} \Vert_{L^p} \Vert \widehat{H} \Vert_{L^q} \cdot 2^{C_0 M_0} , 
  \qquad i=2,3,  %\, \ast \in \lbrace S,R \rbrace.
\\
\label{mainbilinRRe}
& \big \Vert P_K T^{Re}[b](G,H) \big \Vert_{L^r} , \big \Vert P_K T_1^{R,i}[b](G,H) \big \Vert_{L^r} 
  \lesssim \Vert \widehat{G} \Vert_{L^p} \Vert \widehat{H} \Vert_{L^q} \cdot 2^{C_0 A}, \ \ i=1,2.
\end{align}
\end{subequations}
\end{theorem}

\begin{remark}[Bilinear estimates]
Note that the bilinear estimates in \eqref{mainbilin} 
differ from those of \cite{LP} since we allow for the endpoint $r=1.$
Note that it is possible at the expense of adding a projection $P_K$ in front, see $(6.35)$ in \cite{PS}. 
This is harmless since our estimates are done in $L^{\infty}_{\xi}.$ 
\end{remark}

\begin{remark}[Estimating $\mathcal{D}$]\label{remD}
In the bilinear bound \eqref{mainbilinS1} the quantity $\mathcal{D}$ appears. 
This is a lower order remainder term in all our estimates,
and is already adequately estimated in \cite{LP};
therefore we can disregard these terms in our proof.
\end{remark}

\section{Bilinear estimates}\label{secbil}

This section contains the new bilinear estimates we need for the proof of Proposition \ref{main}.

\begin{lemma}\label{lembil}
Let $m=m(\xi,\eta)$, $\xi,\eta\in\R^3$ be a bounded multiplier localized at frequencies 
$\vert \xi \vert \approx 2^K, \vert \eta \vert \approx 2^{L}$ such that
%Assume moreover
\begin{align}\label{boundm}
\big \vert \nabla_{\xi}^{\alpha} \nabla_{\eta}^{\beta} m(\xi,\eta) \big \vert 
  \lesssim 2^{-\vert \alpha \vert K - \vert \beta \vert L}, \ \ 
  0 \leqslant \vert \alpha \vert +\vert \beta \vert \leqslant 4,
\end{align}
and let $b_0=b_0(\omega,\xi)$, $\omega\in\mathbb{S}^2,\xi\in\R^3$, be such that 
\begin{align} \label{boundg}
\big \vert %\nabla_{\omega}^{\alpha} 
  \nabla_{\xi}^{\beta} b_0(\omega,\xi) \big \vert 
  \lesssim 1 %\langle \xi \rangle^{\vert \alpha \vert}  
  + \big( \vert \xi \vert / \langle \xi \rangle \big )^{1-\vert \beta \vert}, 
  \ \ 0 \leqslant %\vert \alpha \vert + 
  \vert \beta \vert \leqslant 4.
\end{align}
We have (note the frequency localization on the profile $h$)
\begin{align}\label{newbilin1}
\begin{split}\Bigg \Vert \int_{\mathbb{R}^6} \wt{f}(\xi-\eta) 
  (\varphi_M\wt{h})(\sigma) m(\xi,\eta) b_0(\eta/\vert \eta \vert,\s) 
  \pv \frac{1}{\vert \eta \vert - \vert \sigma \vert} d\eta d\sigma \Bigg \Vert_{L^{\infty}_{\xi}} 
  \\
  \lesssim %2^{L+M} \langle 2^{L} \rangle^{-1} \langle 2^{M} \rangle^{-1} 
  %\Vert m \Vert_{L^{\infty}_{\xi,\eta}} \Vert f \Vert_{H^2} \Vert h \Vert_{H^2},
  %\qquad 
  2^{L+M} \langle 2^{M} \rangle^{-2} 
  \Vert m \Vert_{L^{\infty}_{\xi,\eta}} \Vert f \Vert_{L^2} \Vert h \Vert_{H^2},
\end{split}
\end{align}
and
\begin{align}\label{newbilin2}
\Bigg \Vert \int_{\mathbb{R}^6} \wt{f}(\eta + \sigma) \wt{h}(\sigma) m(\eta,\sigma) 
  \pv \frac{1}{\vert \eta \vert - \vert \xi \vert} b_0 (\eta/\vert \eta \vert,\xi) 
  d\eta d\sigma \Bigg \Vert_{L^{\infty}_{\xi}}
  \lesssim 2^{2K} \langle 2^K \rangle^{-2}  \Vert f \Vert_{H^2} \Vert h \Vert_{H^2}.
\end{align}
\end{lemma}

\begin{proof}
Let $\epsilon \in (0,1).$ We stress that the implicit constants in the argument below do not depend on $\epsilon.$ 
Denote $\varphi_{\epsilon}$ a radial cut-off function supported on the region $\epsilon < r < \epsilon^{-1}.$ 

We first prove the more challenging inequality \eqref{newbilin2}. 
We start by inserting a cut-off $\varphi_k(\xi)$ and pass to polar coordinates 
%$\xi = r \omega, \eta = \rho \phi:$
$\xi = \rho \phi$ and $\eta = r \omega$, $\phi,\omega \in \mathbb{S}^2$, and estimate
the left-hand side of \eqref{newbilin2} by
\begin{align}
%\nonumber
%&\Bigg \Vert \int_{\mathbb{R}^6} \wt{f}(\eta + \sigma) \wt{h}(\sigma) m(\eta,\sigma) \pv \frac{1}{\vert \eta \vert - \vert \xi \vert} b_0 (\eta/\vert \eta \vert,\xi) d\eta d\sigma \Bigg \Vert_{L^{\infty}_{\xi}} \\
%& \sup_{k \in \mathbb{Z}} \Bigg \vert \int_{\mathbb{R}^6} \wt{f}(\eta + \sigma) 
%  \wt{h}(\sigma) m(\eta,\sigma) \pv \frac{1}{\vert \eta \vert - \vert \xi \vert} \varphi_{k}(\xi) 
%  g (\eta/\vert \eta \vert,\xi) d\eta d\sigma \Bigg \vert    
%\\ 
\label{pr1}
%&\lesssim 
\sup_{\epsilon>0} \sup_{k \in \mathbb{Z}} \sup_{\phi \in \mathbb{S}^2,\rho > 0}
  \Big| \int_{\mathbb{S}^2} \int_0^{\infty} \int_{\mathbb{R}^3} \wt{f}(r\omega+\sigma) m(r \omega, \sigma) 
  \wt{h}(\sigma) d\sigma \, \frac{\varphi_{\epsilon}(r-\rho)}{r-\rho}  \varphi_{k}(\rho) 
  b_0(\omega,\rho \phi) r^2 dr d\omega \Big|.
%\\
%& \lesssim \sup_{\delta>0} \sup_{k \in \mathbb{Z}} \sup_{\phi \in \mathbb{S}^2} \sup_{\rho > 0}
%\Bigg \vert \int_{\mathbb{S}^2} \int_0^{\infty} \int_{\mathbb{R}^3} \langle r \omega + \sigma \rangle^2 
%\wt{f}(r\omega+\sigma) \frac{m(r \omega, \sigma)}{\langle r \omega + \sigma \rangle^2 \langle \sigma \rangle^2} 
%\langle \sigma \rangle^2 \wt{h}(\sigma) d\sigma \frac{\varphi_{\delta}(r-\rho)}{r-\rho} 
%\varphi_{k}(\rho) b_0(\omega,\rho \phi) r^2 dr d\omega \Bigg \vert .
\end{align}
Let us define $f_1 := L^2 f$ and $h_1 := L^2 h$, and
\begin{align*}
& F(r,\omega):= \int_{\mathbb{R}^3} \wt{f_1}(r\omega + \sigma) m_1 (r\omega,\sigma) \wt{h_1}(\sigma) d\sigma,
%\\
%& f_1 := %\langle \vert \nabla \vert \rangle^2 
%\fp{L^2} f, 
%\quad h_1 := %\langle \vert \nabla \vert \rangle^2 
%\fp{L^2} h,
\quad m_1 (\eta,\sigma) := \frac{m(\eta,\sigma)}{\langle \eta + \sigma \rangle^{2} \jsig^{2} }.
\end{align*}
%From now on we denote $f_1 := \langle \vert \nabla \vert \rangle^2 f, 
%h_1 := \langle \vert \nabla \vert \rangle^2 h, m_1 (\eta,\sigma) 
%= \langle \eta + \sigma \rangle^{-2} \langle \sigma \rangle^{-2} m(\eta,\sigma).$
We then take an inverse Fourier transform in $\rho$ in \eqref{pr1}; 
using the fact that the Fourier transform of the Hilbert transform kernel is bounded, 
we can bound \eqref{pr1} by a constant times
\begin{align}
\nonumber
& \sup_{\epsilon>0} \sup_{k \in \mathbb{Z}} \sup_{\phi \in \mathbb{S}^2} 
  \Bigg \Vert \int_{\mathbb{S}^2}  \widehat{\mathcal{F}}^{-1}_{\rho \rightarrow x} 
  \Bigg\{  \Bigg[ %\int_{\mathbb{R}} \Bigg[ 
  \int_{\mathbb{R}} r^2 F(r,\omega) \textbf{1}_{[0,\infty)}(r)
  \frac{\varphi_{\epsilon}(r-\rho)}{r-\rho} dr \Bigg] \varphi_{k}(\rho) b_0(\omega,\rho\phi) % d\rho \Bigg]  
  \Bigg\} d\omega  \Bigg \Vert_{L^1_x(\mathbb{R})} 
  \\
  \nonumber
&\lesssim \sup_{\epsilon>0} \sup_{k \in \mathbb{Z}} \sup_{\omega,\phi \in \mathbb{S}^2} 
  \bigg \Vert \widehat{\mathcal{F}}^{-1}_{\rho \rightarrow x} \bigg( \frac{\varphi_{\epsilon}(r)}{r} 
  \ast \big(r^2 F(r,\omega) \textbf{1}_{[0,\infty)}(r) \big)  \bigg) \bigg \Vert_{L^1_x(\mathbb{R})} 
  {\big\| \mathcal{F}^{-1}_{\rho \rightarrow x} \big( \varphi_{k}(\cdot) b_0(\omega, \cdot \phi) \big)
  \big\|}_{L^1_x(\mathbb{R})} 
  \\
& \lesssim  \bigg \Vert \widehat{\mathcal{F}}^{-1}_{r \rightarrow x} \big( r^2 F(r,\omega) 
  \textbf{1}_{[0,\infty)}(r) \big) \bigg \Vert_{L^{\infty}_{\omega,\phi} L^1_x} 
  \cdot  \sup_{k \in \mathbb{Z}} \bigg \Vert \widehat{\mathcal{F}}^{-1}_{\rho \rightarrow x} 
  \big( \varphi_{k}(\cdot) b_0(\omega, \cdot \phi) \big) \bigg \Vert_{L^{\infty}_{\omega,\phi} L^1_x}.
\label{pr2}
\end{align}

To bound the second term in \eqref{pr2} we use the pointwise bound
\begin{align*}
\big \vert \widehat{ \mathcal{F}}^{-1}_{\rho \rightarrow x} \big( \varphi_k (\cdot) b_0(\omega, \cdot \phi) \big) (x) \big \vert \lesssim \frac{2^k}{(1 + \vert x \vert)(1 + 2^k/\langle 2^k \rangle \vert x \vert )^2} ,
\end{align*}
which follows by integration by parts in $\rho$ and our assumptions \eqref{boundg} on $b_0.$ 
This yields
\begin{align*}
 \bigg \Vert \widehat{ \mathcal{F}}^{-1}_{\rho \rightarrow x} \big( \varphi_{k}(\cdot) 
 b_0(\omega, \cdot \phi) \big) \bigg \Vert_{L^{\infty}_{\omega,\phi} L^1_x} \lesssim \langle 2^k \rangle . 
\end{align*}

We then focus on the first term. 
We denote the wave operator by $\mathcal{W} := \wt{\mathcal{F}}^{-1} \widehat{\mathcal{F}}, 
\mathcal{W}^{*} = \widehat{\mathcal{F}}^{-1} \widetilde{\mathcal{F}},$
and recall that it is bounded on Sobolev spaces under our assumptions.
We change coordinates by denoting $R_{\omega}$ the rotation %centered at the origin 
that sends $\omega$ to $e_1:=(1,0,0)$ and denoting 
$g^\omega := g\circ R_\omega^{-1}$,
%$f_{1}^{\omega} := f_1 \circ R_{\omega}^{-1}, h_{1}^{\omega} := h_1 \circ R_{\omega}^{-1},$ 
so that we find (omitting irrelevant constants)
\begin{align*}
& \widehat{\mathcal{F}}^{-1}_{r \rightarrow x} \big( r^2 F(r,\omega) \textbf{1}_{[0,\infty)}(r) \big) 
  = \int_{\mathbb{R}}e^{ixr} \int_{\mathbb{R}^3} \wt{f_1}(\sigma +r\omega) 
  m_1(r\omega,\sigma) \wt{h_1}(\sigma) d\sigma \, r^2 \textbf{1}_{[0,\infty)}(r)  dr 
  \\
&= \int_{\mathbb{R}^6} \Bigg[ \int_{\mathbb{R}^4} e^{ir(x-y\cdot \omega)} 
  e^{-i\sigma \cdot(y+z)} m_1(r \omega, \sigma) r^2 \textbf{1}_{[0,\infty)}(r) dr d\sigma \Bigg]
  \mathcal{W}^{*}f_1(y) \mathcal{W}^{*}h_1(z)   dy dz 
  \\
&:= \int_{\mathbb{R}^6} T[m_1](x-y_1, R_{\omega}^{-1} y + R_{\omega}^{-1} z) 
 (\mathcal{W}^{*}f_1)^{\omega}(y) (\mathcal{W}^{*}h_1)^{\omega}(z) dy dz.
\end{align*}
Note that the last identity defines $T[m_1]$.
We then observe that, given the localization of $m_1(\eta,\s)$ to $|\eta|\approx 2^K$ and $|\s|\approx 2^L$,
using integration by parts and \eqref{boundm} we have
%To obtain the last inequality we used the pointwise bound
\begin{align*}
\big|  T[m_1](y_1,z) \big| = \Bigg \vert  \int_{\mathbb{R}^4} e^{i y_1 r } 
e^{-i\sigma \cdot z} m_1(r \omega, \sigma) r^2 \textbf{1}_{[0,\infty)}(r) dr d\sigma \Bigg \vert 
  \lesssim \frac{2^{3 K} \langle 2^K \rangle^{-2}}{\big(1 + 2^K \vert y_1 \vert \big)^2} 
  \frac{2^{3L} \langle 2^L \rangle^{-2}}{\big(1 + 2^L \vert z \vert \big)^4}.
\end{align*}
This yields the bound
\begin{align*}
\big \Vert T[m_1]%(y_1 \omega , z) 
  \big \Vert_{L^{1}(\R\times\R^3)} \lesssim 2^{2K} \langle 2^K \rangle^{-2}.
\end{align*}
Changing variables $z \mapsto y-z, y \mapsto (x-y_1,y_2,y_3)$ and using the Cauchy-Schwarz inequality as well as the boundedness of wave operators on $L^2$, we have
\begin{align*}
& \big \vert \widehat{\mathcal{F}}^{-1}_{r \rightarrow x} 
  \big( r^2 F(r,\omega) \textbf{1}_{[0,\infty)}(r) \big) \big \vert 
  \lesssim \int_{\mathbb{R}^4} \big \vert T[m_1](y_1 %\omega
  , R_{\omega}^{-1} z) \big \vert dy_1 dz 
  \\
  & \times \sup_{(y_1,z) \in \mathbb{R}^4} \int_{\mathbb{R}^3} \big \vert 
  (\mathcal{W}^{*} f_1)^{\omega}(x-y_1,y_2,y_3) 
  \big \vert 
  \cdot \big \vert (\mathcal{W}^{*} h_1)^{\omega}(x-y_1-z_1,y_2-z_2,y_3-z_3) \big \vert dx dy_2 dy_3 
  \\
  & \lesssim 2^{2K} \langle 2^K \rangle^{-2} {\| f_1\|}_{L^2} {\| h_1\|}_{L^2}
  \lesssim \Vert f \Vert_{H^2} \Vert h \Vert_{H^2}.
\end{align*}

We now prove the easier inequality \eqref{newbilin1}. 
Passing to polar coordinates $\eta = r \omega$ and $\s = \rho \phi,$ $\omega, \phi \in \mathbb{S}^2$, 
we estimate using the boundedness of the Hilbert transform on $L^2,$ and recalling that
on the support of the integral $|\xi|\approx 2^K$, $|\eta|\approx 2^L$ and $|\s| \approx 2^M$:
\begin{align*}
& \Bigg \vert \int_0^{\infty} \int_{\mathbb{S}^2} \wt{f}(\xi-r \omega) m(\xi,r\omega) \int_{0}^{\infty}
  \int_{\mathbb{S}^2} b_0(\omega,\rho \phi) \frac{\varphi_{\epsilon}(r-\rho)}{r-\rho} \wt{h}(\rho \phi)
    \varphi_{M}(\rho) \rho^2 d\rho d\phi \, r^2 dr d\omega \Bigg \vert 
    \\
& \lesssim \int_{\mathbb{S}^2} 
  %\bigg[ \int_{0}^{\infty} \big \vert \wt{f}(\xi-r \omega) m(\xi,r \omega) \big \vert^2 r^4 dr \bigg]^{1/2} 
  {\big\| \wt{f}(\xi-r \omega) m(\xi,r \omega) \, r \|}_{L^2_r(r^2dr)} d\omega
  %\\
  %& \times 
  \cdot \sup_{\omega \in \mathbb{S}^2} \bigg \Vert \frac{\varphi_{\epsilon}(r)}{r}
  \ast \bigg( \int_{\mathbb{S}^2} b_0(\omega, r \phi) \wt{h}(r \phi) %\frac{\wt{h}(r \phi)}{\langle r \rangle} 
  \varphi_{M}(r) r^2 
  d\phi \bigg)
  \bigg \Vert_{L^2_r} 
  \\
& \lesssim {\| m \|}_{L^\infty_{\xi,\eta}} 
  \cdot 2^L  {\big\| \varphi_L(\cdot) \wt{f}(\xi-\cdot) \big\|}_{L^2} 
  \cdot 2^{M} {\big\| \varphi_M \wt{h} \big\|}_{L^2(\R^3)}.
%& \lesssim \frac{2^{L}}{\langle 2^{L} \rangle} (2^{2L} + 2^{2K}) 
%   \Vert f \Vert_{H^2} \cdot \frac{2^{M}}{\langle 2^{M}\rangle^{2}} {\| h \|}_{H^2}.
%\\
%& \lesssim 2^{L} \langle 2^{L} \rangle^{-1} \Vert f \Vert_{H^2} \cdot 2^{M} \langle 2^{M} 
%\rangle^{-1} {\| h \|}_{H^2}.
\end{align*}
\end{proof}

\section{Asymptotic behavior of $h$}
In this section we begin the verification of the hypotheses in Lemma \ref{Nexist}. %(which implies the main Proposition \ref{main}).
We start by reducing these to proving a set of bounds for the bilinear and trilinear terms appearing
in Lemma \ref{lemhdecomp}.

\begin{lemma}\label{lemscatteringh}
Recall the decomposition for $h$ in \eqref{decomph}-\eqref{FHL} (see also the notation for $F_{H,m}$ below \eqref{FHL}), 
and for $F_L$ in \eqref{decompFL}-\eqref{Reint}.
Let (see \eqref{Sbdry})
\begin{align}\label{defN}
\mathcal{N}(\wt{f})(t,\xi) := F^{S,(1)}(t,\xi). 
\end{align}
Assume a priori that, for all $t\geqslant 0$,
\begin{align}\label{bootLinfty}
%\sup_{t \in [0,T]} 
{\big\| \wt{h}(t) \big\|}_{L^{\infty}_{\xi}} \leqslant \varepsilon^{\delta},
\end{align}
and that the following bounds hold for all $m = 0,1,\dots, L+1$ (recall \eqref{timedecomp}):
\begin{align} \label{scatteringh}
\Vert S_m \Vert_{L^{\infty}_{\xi}}, \, \Vert M_m \Vert_{L^{\infty}_{\xi}}, \, \Vert F_{H, m} \Vert_{L^{\infty}_{\xi}}, 
  \,
  \Vert F^{S,(2)}_m \Vert_{L^{\infty}_{\xi}}, \, \Vert F^{R}_m \Vert_{L^{\infty}_{\xi}}, 
  \, \Vert F^{Re}_m \Vert_{L^{\infty}_{\xi}}
  \lesssim 2^{-2\delta m} \varepsilon^{\delta}.
\end{align}

Then the following hold true:

\setlength{\leftmargini}{1.5em}
\begin{itemize}

\item[(i)] There exists $h_\infty \in \e^\delta L^\infty_\xi$ such that \eqref{hinfty} holds true.

\item[(ii)] We have
\begin{align}\label{corrh}
{\big\| \mathcal{N}(\wt{f})(t) \big\|}_{L^{\infty}_t L^{\infty}_{\xi}} \lesssim \varepsilon^{2\delta}.
\end{align} 

\item[(iii)] The a priori assumption \eqref{bootLinfty} can be improved to
\begin{align} \label{bootLinftyconc}
%\sup_{t \in [0,T]} 
{\big\| \wt{h}(t) \big\|}_{L^{\infty}_{\xi}} \leqslant \tfrac{1}{2}\varepsilon^{\delta}.
\end{align}

\end{itemize}
\end{lemma}

The main goal of this section is to prove
the estimates \eqref{scatteringh}; these proofs rely on the machinery developed in \cite{LP}.
We will also use the following:

\begin{remark}\label{loctime}
Note that obtaining a bound by  $2^{-3\delta m}$ for all the quantities on the left-hand side of 
\eqref{scatteringh} is sufficient since, as in \cite{LP}, we may assume
that we are working past the local time of existence of $O(1/\varepsilon).$
\end{remark}

Below we show how (i) and (iii) in Lemma \ref{lemscatteringh} follow from the assumptions
\eqref{bootLinfty} and \eqref{scatteringh}, using (ii) as well.
The bound \eqref{corrh} in (ii) %(which is stronger than what is assumed in Lemma \ref{Nexist}) 
is proved in Section \ref{seccorr} relying on the
new bilinear bounds from Section \ref{secbil} and the a priori assumption \eqref{bootLinfty}.

%we show that $h$ corrected by a quadratic normal form has a limit in $L^{\infty}_{\xi}$ 
%of size $\varepsilon^{\delta}$ (for some small $\delta$).  

\begin{proof}[Proof of (i) and (iii) in Lemma \ref{lemscatteringh}]
%To show \eqref{hinfty} it suffices to prove that $h(t) - \mathcal{N}(\wt{f})(t)$
%is Cauchy in time with values in $L^\infty_\xi$.
From the formulas \eqref{decomph}, \eqref{FHL} (see also the notation for $F_{H,m}$ below that), 
and \eqref{decompFL}, we see that \eqref{scatteringh} implies, for all $t_1<t_2$,
\begin{align*}
{\big\| \wt{h}(t_1) - \mathcal{N}(\wt{f})(t_1) 
  - \big( \wt{h}(t_2) - \mathcal{N}(\wt{f})(t_2) \big) \big\|}_{L^\infty_\xi} 
  \lesssim \e^\delta \langle t_1 \rangle^{-2\delta}.
\end{align*}
Then $\wt{h}(t) - \mathcal{N}(\wt{f})(t)$ is Cauchy in time with values in $L^\infty_\xi$
and the first claim follows. %converges to a limit $h_\infty \in \e^\delta L^\infty_\xi$.
%The inequality \eqref{hinfty} follows by setting $h_\infty := $

For the bound \eqref{bootLinftyconc} we use again the formula for $h$
given by \eqref{decomph} with \eqref{FHL} and \eqref{decompFL},
and the bounds on each of the terms appearing in this formula
provided by \eqref{scatteringh} (under the assumption that $2^m \gtrsim \e^{-1}$, see Remark \ref{loctime})
and \eqref{corrh} (see the definition \eqref{defN}).
\end{proof}

To summarize, in order to prove Lemma \ref{lemscatteringh} and obtain our main 
result in Proposition \ref{main},
it suffices to prove all the bounds in \eqref{scatteringh}, 
and the bound \eqref{corrh} under the assumption \eqref{bootLinfty}.

%The rest of this section is dedicated to proving all the estimates in \eqref{scatteringh}.
%These proof rely on the machinery developed in \cite{LP}.
%The final bound \eqref{corrh} is proved in the last Section \ref{seccorr}.

%\begin{remark}
%Note that the estimates \eqref{scatteringh} also show that $\big( \wt{h}(t) - \mathcal{N}(\wt{f}) (t) \big)$ 
%is Cauchy in $L^{\infty}_{\xi},$ and therefore converges to some 
%$h_{\infty} \in \varepsilon^{\delta} L^{\infty}_{\xi}.$
%\cfp{subtract $t=0$ in \eqref{hinfty}?}
%\end{remark}

\begin{remark}
We remark that the uniform bound  \eqref{bootLinfty} for $\wt{h}$ is not proven in \cite{LP}.
Moreover, as a consequence of \eqref{bootLinfty} and \eqref{gLinftyO(1)} we also have, for all $t\geqslant 0$,
\begin{align} \label{fLinfty}
%\sup_{t \in [0,T]} 
\Vert \wt{f}(t) \Vert_{L^{\infty}_{\xi}} \lesssim 1.
\end{align}
\end{remark}

%\begin{proof}[Proof of \eqref{scatteringh}]
%We proceed in a few steps

\smallskip
\noindent
{\bf Proof of \eqref{scatteringh}.}
We estimate each of the terms on the right-hand side of \eqref{scatteringh} by $2^{-3\delta m}$.
%We will often omit the indexes $\eps_1,\eps_2$ in the notation as they do not play any role
We will constantly use the bilinear bounds stated in Theorem \ref{bilinearmeas} (see also Remark \ref{remD})
without referring to them at each application.

\smallskip
\noindent
{\bf Step 1: Bounding $F^{S,(2)}_m.$} 
%Dealing next with the integrated terms,
From \eqref{Sint} we have that
\begin{align*}
& \Vert F^{S,(2)}_m \Vert_{L^{\infty}_{\xi}} \lesssim 2^{m} \cdot \sup_{k \in \mathbb{Z}} 
  \sum_{G,H \in \lbrace g,h \rbrace} \sum_{k_1,k_2 \leqslant \delta_N m} \sup_{s\approx 2^m} 
  \Big[ {\big\| \varphi_k \widehat{\mathcal{F}}_{x \rightarrow \xi} T_0[b]\big( e^{\eps_1 is\jxi}\partial_s\wt{G_{\eps_1}}, 
    e^{\eps_2 is\jxi}\wt{H_{\eps_2}} \big) \big\|}_{L^{\infty}_{\xi}} 
  \\
  & + {\big\|  \varphi_k \widehat{\mathcal{F}}_{x \rightarrow \xi} T_1^{S,1}[b_1]\big( e^{\eps_1 is\jxi}\partial_s\wt{G_{\eps_1}}, 
    e^{\eps_2 is\jxi}\wt{H_{\eps_2}} \big) \big\|}_{L^{\infty}_{\xi}}  + {\big\| \varphi_k \widehat{\mathcal{F}}_{x \rightarrow \xi} T_1^{S,1}[b_1']\big( e^{\eps_2 is\jxi}\wt{H_{\eps_2}},  e^{\eps_1 is\jxi}\partial_s \wt{G_{\eps_1}} \big) \big\|}_{L^{\infty}_{\xi}} 
\\   
    &+ {\big\|  \varphi_k \widehat{\mathcal{F}}_{x \rightarrow \xi} T_1^{S,2}[b_2]\big( e^{\eps_1 is\jxi} \partial_s \wt{G_{\eps_1}}, 
    e^{\eps_2 is\jxi}\wt{H_{\eps_2}} \big) \big\|}_{L^{\infty}_{\xi}} + \sum_{i=2,3} {\big\|  \varphi_k \widehat{\mathcal{F}}_{x \rightarrow \xi} T_i^S[b]\big( e^{\eps_1 is\jxi}\partial_s\wt{G_{\eps_1}}, 
    e^{\eps_2 is\jxi}\wt{H_{\eps_2}} \big) \big\|}_{L^{\infty}_{\xi}} \Big] .
\end{align*} 
Since the sum over $k_1,k_2$ will only contribute $O(m)$ terms we will often disregard them.

\smallskip
\noindent
{\it Case 1: $G=g$.} We estimate 
using \eqref{mainbilinS1}, the decay estimates \eqref{basic-bounds-decay} and \eqref{time-der-g}, 
that
\begin{align*}
{\big\|  \varphi_k \widehat{\mathcal{F}}_{x \rightarrow \xi} T_1^{S,1}[b_1]
  \big( e^{\eps_1 is\jxi}\partial_s\wt{g_{\eps_1}}, e^{\eps_2 is\jxi}\wt{H_{\eps_2}} \big)
  \big\|}_{L^{\infty}_{\xi}} & \lesssim  {\big\| P_k T_1^{S,1}[b_1]
  \big( e^{\eps_1 is\jxi}\partial_s\wt{g_{\eps_1}}, e^{\eps_2 is\jxi}\wt{H_{\eps_2}} \big) \big\|}_{L^1_x}
\\ 
  & \lesssim {\big\| e^{isL} \partial_s g \big\|}_{L^{6/5}_x} \cdot 
    {\| e^{isL} P_{k_2}H \|}_{L^6_x} \cdot 2^{C_0 M_0} 
  \\
  & \lesssim  \rho(2^m) \cdot 2^{-m + 10 M_0} Z \cdot 2^{C_0 M_0} .
\end{align*}
This is an acceptable bound.

\smallskip
\noindent
{\it Case 2: $G=h$.} 
Using \eqref{mainbilinS1}, \eqref{basic-bounds-decay} and \eqref{time-der-h}, we estimate similarly that
\begin{align}
\begin{split}
 {\big\| P_k T_1^{S,1}[b_1]\big( e^{\eps_1 is\jxi}\partial_s\wt{h_{\eps_1}}, 
 e^{\eps_2 is\jxi}\wt{H_{\eps_2}} \big) \big\|}_{L^1_x}  
 & \lesssim {\big\| e^{isL} \partial_s h \big\|}_{L^{3/2-}_x} \cdot {\| e^{isL} P_{k_2}H \|}_{L^3_x} \cdot 2^{C_0M_0} 
 \\
    &\lesssim 2^{-m + 10 M_0} Z^2 \cdot 2^{-m/2+10 M_0} Z \cdot 2^{C_0 M_0},
\end{split}
\end{align}
which yields a sufficient bound.

\smallskip
\noindent
{\bf Step 2: Bounding $F^R_m, F^{Re}_m.$} 
From \eqref{muReest} and \eqref{nuRest} one can see that $F^{Re}_m$ satisfies better estimates than $F^R_m,$ 
and is therefore strictly easier to treat. 
Therefore we only explain how to deal with $F^R_m.$ 
Starting with the first piece we write that, using \eqref{Rint},
\begin{align*}
\begin{split}
\Vert F^{R,(1)}_m \Vert_{L^{\infty}_{\xi}}  & \lesssim 2^m \sum_{G,H \in \lbrace g,h \rbrace}  
  \sum_{k_1,k_2 \leqslant \delta_N m} \sup_{s \approx 2^m}
  {\big\|  \varphi_k \widehat{\mathcal{F}}_{x \rightarrow \xi} T^{R,1}_1[b_1] (e^{is \eps_1 \jxi} \wt{G_{\eps_1}}, 
  e^{is \epsilon_2 \jxi} \wt{H_{\eps_2}}) \big\|}_{L^{\infty}_{\xi}} 
\\
  & + {\big\|  \varphi_k \widehat{\mathcal{F}}_{x \rightarrow \xi} T^{R,1}_1 [b_1'] (e^{is \epsilon_2 \jxi} \wt{H_{\eps_2}},
  e^{is \eps_1 \jxi} \wt{G_{\eps_1}}) \big\|}_{L^{\infty}_{\xi}} ;
\end{split}
\end{align*}
Recall the definition for the symbols \eqref{Rintm} 
with the notation \eqref{symnot}, and note that these always localize
the integration variables $\eta,\s$ in the formulas \eqref{idop}
to $|\eta|\approx2^{k_1}$, $|\s| \approx 2^{k_2}$, but do not localize the inputs of the operators necessarily.
It suffices to estimate the term $T^{R,1}_1[b_1]$.

\smallskip
{\it Case 1: $k_2 \leqslant -m/2 + \delta m.$}
%Localization on 
Then integrating directly we find that, using \eqref{basic-bounds-hardy} and \eqref{nuRest},
\begin{align*}
& 2^m \sup_{s \approx 2^m}  {\big\|  \varphi_k \widehat{\mathcal{F}}_{x \rightarrow \xi} T^{R,1}_1[b_1] (e^{is \eps_1 \jxi} \wt{G_{\eps_1}}, 
  e^{is \epsilon_2 \jxi}  \wt{H_{\eps_2}}) \big\|}_{L^{\infty}_{\xi}}  \\
& \lesssim 2^m \Vert \varphi_{k_1}(\cdot -\eta) \wt{G}(s,\eta) \Vert_{L^1_{\eta}} \cdot \Vert \varphi_{k_2} \wt{H}\Vert_{L^1_{\s}} \cdot \sup_{\vert \eta \vert \approx 2^{k_1}, \vert \s \vert \approx 2^{k_2}} \vert \nu^R_1(\eta,\sigma) \vert  \\
& \lesssim 2^{m} \cdot 2^{(5/2-)k_1 } Z \cdot 2^{(5/2-) k_2}Z \cdot 2^{-2(k_1 \vee k_2)} 2^{10 M_0}.
\end{align*}
This gives the desired bound.

\smallskip
{\it Case 2: $ k_2 \geqslant -m/2 + \delta m$ %}
%\smallskip
%{\it Subcase 2.1: 
and $k_1 \leqslant -m/10.$}  

\smallskip 
\noindent
%{\it Subcase 2.1.1: 
{\it Case $H=h.$} 
Integrating by parts in $\s$ and  
using \eqref{basic-bounds-hardy} and \eqref{nuRest} 
(note that the case where the derivative hits the cut-offs or measure are better and can be disregarded)
we obtain the bound
\begin{align*}
& 2^m \sup_{s \approx 2^m}  {\big\|  \varphi_k \widehat{\mathcal{F}}_{x \rightarrow \xi} T^{R,1}_1[b_1] (e^{is \eps_1 \jxi} \wt{G_{\eps_1}}, 
  e^{is \epsilon_2 \jxi}  \wt{H_{\eps_2}}) \big\|}_{L^{\infty}_{\xi}}  \\
& \lesssim \sup_{s \approx 2^m}  \Vert \varphi_{k_1}(\cdot - \eta) \wt{G}(s,\eta) \Vert_{L^1_{\eta}} \cdot \bigg \Vert \varphi_{k_2}(\sigma) \big( \frac{\sigma}{\vert \sigma \vert^2} \cdot \nabla_{\s} \big) \wt{h} \bigg \Vert_{L^1_{\s}} \cdot \bigg( \sup_{\vert \eta \vert \approx 2^{k_1}, \vert \s \vert \approx 2^{k_2}} \vert \nu_1^R(\eta,\s) \vert \bigg)  \\
& \lesssim 2^{(5/2-) k_1} Z \cdot 2^{k_2/2} Z \cdot 2^{-2 (k_1 \vee k_2)} 2^{10 M_0},
\end{align*}
which is sufficient to conclude.

\smallskip
\noindent
{\it %Subcase 2.1.2: 
Case $H=g.$} 
In this case note first that $k_2 \approx 0.$ 
We let $\ell_0 := \lfloor -m+\delta m \rfloor$
and insert cut-offs $\varphi_{\ell}^{(\ell_0)}(\jsigma - 2 \lambda)$ 
defined for $\ell \geqslant \ell_0$ 
that are such that $\varphi_{\ell_0}^{(\ell_0)} = \varphi_{\leqslant \ell_0}$ 
and $\varphi_{\ell}^{(\ell_0)}=\varphi_{\ell}$ if $\ell > \ell_0$.
In the case where $\ell = \ell_0$
we integrate directly and find that, using \eqref{bounds-ginfty}, \eqref{basic-bounds-hardy} and \eqref{nuRest}
\begin{align*}
& 2^m \sup_{s \approx 2^m}  {\big\|  \varphi_k \widehat{\mathcal{F}}_{x \rightarrow \xi} T^{R,1}_1[b_1] (e^{is \eps_1 \jxi} \wt{G_{\eps_1}}, 
  e^{is \epsilon_2 \jxi}  \varphi_{\ell}^{(\ell_0)} (\jxi -2 \lambda) \wt{g_{\eps_2}}) \big\|}_{L^{\infty}_{\xi}}
  \\
& \lesssim 2^m \cdot \sup_{s \approx 2^m} \Vert \varphi_{k_1}(\cdot -\eta) \wt{G}(s,\eta) \Vert_{L^1_{\eta}} \cdot 2^{\ell_0} \Vert \wt{g} \Vert_{L^{\infty}_{\sigma}} \cdot 2^{10 M_0}  \lesssim 2^m \cdot 2^{(5/2-)k_2} Z \cdot 2^{\ell_0} \rho(2^m) 2^m m \cdot 2^{10 M_0},
\end{align*}
which yields an acceptable bound provided $\delta>0$ is chosen small enough.
If $\ell>\ell_0$ we can integrate by parts in $\sigma$ and using \eqref{bounds-g-derL1}, 
\eqref{basic-bounds-hardy} and \eqref{nuRest} we obtain
\begin{align*}
& 2^m \sup_{s \approx 2^m}  {\big\|  \varphi_k \widehat{\mathcal{F}}_{x \rightarrow \xi} T^{R,1}_1[b_1] (e^{is \eps_1 \jxi} \wt{G_{\eps_1}}, 
  e^{is \epsilon_2 \jxi}  \wt{g_{\eps_2}}) \big\|}_{L^{\infty}_{\xi}} \\
& \lesssim \sum_{\ell>\ell_0}  \sup_{s \approx 2^m} \Vert \varphi_{k_1}(\cdot -\eta) \wt{G}(s,\eta) \Vert_{L^1_{\eta}} \cdot \Vert \varphi_{\ell}^{(\ell_0)}(\jsigma - 2 \lambda) \nabla_{\sigma} \wt{g} \Vert_{L^{1}_{\sigma}} \cdot 2^{10 M_0} \lesssim 2^{(5/2-)k_1} Z \cdot 2^m \rho(2^m) m^2 \cdot 2^{10 M_0},
\end{align*}
which is an acceptable bound.

\medskip
{\it Case 3: $k_1 \geqslant -m/10, k_2 \geqslant -m/2 + \delta m.$} 
In this case we have from \eqref{nuRest} that
\begin{align} \label{nuRest-3}
\sup_{\vert \eta \vert \approx 2^{k_1}, \vert \s \vert \approx 2^{k_2}} \vert \nu_1^R(\eta, \sigma) \vert 
  \lesssim 2^{m/10 + 10 M_0}.  
\end{align}
Next we insert another cut-off $\varphi_{k_3}(\xi-\eta)$ to localize the frequency of the input $G$. 

\smallskip
\noindent
{\it Subcase 3.1: $k_3 \leqslant -m/2 + \delta m.$}
We can integrate directly as in Case 1 and obtain, using \eqref{basic-bounds-hardy} and \eqref{nuRest-3},
\begin{align*}
& 2^m \sup_{s \approx 2^m}  {\big\|  \varphi_k \widehat{\mathcal{F}}_{x \rightarrow \xi} T^{R,1}_1[b_1] (e^{is \eps_1 \jxi} \varphi_{k_3}(\xi) \wt{G_{\eps_1}}(s,\xi), 
  e^{is \epsilon_2 \jxi} \varphi_{\sim k_2}(\xi) \wt{H_{\eps_2}}) \big\|}_{L^{\infty}_{\xi}} \\
& \lesssim 2^m \cdot 2^{(5/2-)(k_1 \wedge k_3)} Z \cdot 2^{(5/2-)k_2} Z \cdot 2^{m/10 + 10 M_0}.
\end{align*}
%This is an acceptable bound. 

\smallskip
\noindent
{\it Subcase 3.2: $k_3 >-m/2 + \delta m.$}
In this case we integrate by parts in $\eta$ and $\sigma.$ 
Treating first the case where $(G,H)=(h,h)$ we write, using \eqref{basic-bounds-wh} and \eqref{nuRest-3},
\begin{align*}
& 2^{-m} \sup_{s \approx 2^m}  {\bigg\|  \varphi_k \widehat{\mathcal{F}}_{x \rightarrow \xi} T^{R,1}_1[b_1] 
\big(e^{is \eps_1 \jxi} \varphi_{k_3}(\xi) \big(\frac{\xi}{\vert \xi \vert^2} \cdot \nabla_{\xi} \big) \wt{h_{\eps_1}}, 
  e^{is \epsilon_2 \jxi} \varphi_{\sim k_2}(\xi) \big(\frac{\xi}{\vert \xi \vert^2} \cdot \nabla_{\xi} \big)
  \wt{h_{\eps_2}}\big) \bigg\|}_{L^{\infty}_{\xi}} \\
 & \lesssim 2^{-m}  \cdot 2^{-k_3} 2^{(3/2)(k_3 \wedge k_1)} Z \cdot 2^{k_2/2} Z \cdot 2^{m/10 + 10 M_0}.
\end{align*}
The desired bound follows. 

Moving on to the $(g,g)$ case, we insert cutoffs $\varphi_{\ell_1}^{(\ell_0)}(\langle \xi-\eta \rangle -2 \lambda)$
and $\varphi_{\ell_2}^{(\ell_0)}(\jsigma-2\lambda)$ 
for $\ell_0 = \lfloor -m+\delta m\rfloor$ as above.
Note that $k_3,k_2 \approx 0.$ First we treat the case where $\ell_1,\ell_2>\ell_0.$ 
Then we can integrate by parts in both $\eta$ and $\sigma$ to obtain, using \eqref{bounds-g-derL1}, \eqref{nuRest-3}
\begin{align*}
& 2^m \sup_{s \approx 2^m}  {\bigg\|  \varphi_k \widehat{\mathcal{F}}_{x \rightarrow \xi} T^{R,1}_1[b_1] (e^{is \eps_1 \jxi} \varphi_{\sim 0}(\xi) \wt{g_{\eps_1}}, 
  e^{is \epsilon_2 \jxi} \varphi_{\sim 0}(\xi) \wt{g_{\eps_2}}) \bigg\|}_{L^{\infty}_{\xi}}  \\
& \lesssim 2^{-m} \sum_{\ell_1,\ell_2>\ell_0} \Vert \varphi_{\ell_1}^{(\ell_0)}(\jeta-2 \lambda) \nabla_{\eta} \wt{g_{\eps_1}} \Vert_{L^1_{\eta}} \cdot \Vert \varphi_{\ell_2}^{(\ell_0)}(\jsigma-2\lambda) \nabla_{\sigma} \wt{g_{\eps_2}}) \Vert_{L^1_{\sigma}} \cdot 2^{m/10 + 10 M_0} \\
  & \lesssim 2^{-m} \cdot \big(2^m \rho(2^m) m^2 \big)^2 \cdot 2^{m/10 + 10 M_0} .
\end{align*}
In the case where one of the indices $\ell_i$ equals $\ell_0,$ we can integrate directly 
instead of integrating by parts. We skip the details.

Similarly for the mixed cases $(G,H) =(h,g)$ or $(g,h)$ we can combine the approaches above to obtain the desired bound.

\smallskip
Next, we deal with the piece $F^{R,(2)}_m.$ 

\medskip
\textit{Case 1: $ k_2 \leqslant -m/2 + m/100$.}
Note that we can add a cut-off $\varphi_{\leqslant k_1 \vee k_2 + 5}(\sigma)$ 
on the profile $H.$ %since $\vert \sigma \vert \leqslant \vert \eta-\sigma \vert + \vert \eta \vert.$
Then we can write, using \eqref{basic-bounds-hardy} and \eqref{nuRest}, that
\begin{align*}
& 2^m \sup_{s \approx 2^m} \big \Vert  \varphi_k \widehat{\mathcal{F}}_{x \rightarrow \xi}
T^{R,2}_1[b_2] (e^{is \epsilon_1 \jxi} \wt{G_{\epsilon_1}}, e^{is \epsilon_2 \jxi} \wt{H_{\epsilon_2}} ) 
  \big  \Vert_{L^\infty_{\xi}} 
  \\ & \lesssim 2^m \cdot \Vert \varphi_{k_1}(\cdot  +\eta) \wt{G}(s,\eta) \Vert_{L^1_{\eta}} 
  \Vert \varphi_{k_2} \wt{H_{\epsilon_2}} \Vert_{L^1_{\sigma}} \cdot  2^{-2 k_1} 2^{10 M_0}
\lesssim 2^m \cdot 2^{(5/2-) k_1} Z \cdot 2^{(5/2-)k_2} Z \cdot 2^{-2 k_1} 2^{10M_0}. 
\end{align*}
This is sufficient to conclude.

\medskip
\textit{Case 2: $k_2>-m/2 +m/100$}. 

\smallskip
\textit{Subcase 2.1: $k_1 \leqslant -m/2 + m/200$.}  
In this case $\vert \eta + \sigma \vert \approx 2^{k_2},$ and we can add a cutoff $\varphi_{\sim k_2}(\eta +\sigma)$ to the expression. 

\smallskip
\noindent
\textit{Subcase 2.1.1: $(\epsilon_1,\epsilon_2) \neq (+,-), (-,+)$.}
In this case we note that $\vert \nabla_{\sigma} \Phi_{\epss} \vert \gtrsim 2^{k_2} \langle 2^{k_2} \rangle^{-1}$ 
and, therefore, we can integrate by parts in $\sigma.$ 
We may assume that $k_2<0$, as otherwise the bound is easier to obtain.
Integration by parts gives two types of main terms, one when the derivative hits $H$ and 
another when it hits $G.$
In the first case we integrate directly, in the second we change variables
$\eta \mapsto -\eta - \sigma$ and then integrate. 
Overall, %also taking into account that we do not integrate by parts on a $g$ input when...
using \eqref{basic-bounds-hardy}, 
\eqref{basic-bounds-wh}, \eqref{bounds-g-derL1}, \eqref{bounds-ginfty}, \eqref{nuRest} we get, 
%\cfp{To check. Also, not exactly \eqref{basic-bounds-hardy} ?}
\begin{align*}
& 2^m \sum_{G,H \in \lbrace g,h \rbrace} \sup_{s \approx 2^m} 
  \big \Vert  \varphi_k \widehat{\mathcal{F}}_{x \rightarrow \xi} 
  T^{R,2}_1[b_2] (e^{is \epsilon_1 \jxi} \wt{G_{\epsilon_1}}, e^{is \epsilon_2 \jxi} 
  \wt{H_{\epsilon_2}} ) \big \Vert_{L^\infty_{\xi}} 
  \\
& \lesssim \sum_{H \in \lbrace g,h \rbrace} 2^m \cdot \big \Vert \varphi_{k_1}(%\cdot + 
  \eta) \wt{G}(s,\eta+\sigma) \big \Vert_{L^\infty_\s L^1_{\eta}} 
  \bigg[2^{-m} 2^{-k_2} \big \Vert \varphi_{\sim k_2} 
  \nabla_\s \wt{h} 
  \big \Vert_{L^1_{\sigma}} + 2^{\ell_0} \Vert \wt{g} \Vert_{L^\infty_{\sigma}}  
  \\
&+ \sum_{\ell>\ell_0} 2^{-m} \, \big \Vert \varphi_{\ell}^{(\ell_0)}(\jsigma - 2 \lambda) 
  \nabla_{\sigma} \wt{g} \big \Vert_{L^1_{\sigma}} \bigg] \cdot \sup_{\xi \in \mathbb{R}^3, 
  \vert \eta \vert \approx 2^{k_1}} \vert \nu^R_1 (\xi,\eta) \vert 
  \\
& + \sum_{G \in \lbrace g,h \rbrace} 2^m \cdot
    \big \Vert \varphi_{k_1}(\cdot + \sigma) 
  \wt{H}(s,\sigma) \big \Vert_{L^\infty_\eta L^1_\sigma} 
  \bigg[2^{-m} 2^{-k_2} \big \Vert \varphi_{\sim k_2} 
  \wt{h} \big \Vert_{L^1_{\eta}} + 2^{\ell_0} \Vert \wt{g} \Vert_{L^\infty_{\eta}}  
  \\
& + \sum_{\ell>\ell_0} 2^{-m} \, \big \Vert \varphi_{\ell}^{(\ell_0)}(\jsigma - 2 \lambda) 
  \nabla_{\sigma} \wt{g} \big \Vert_{L^1_{\eta}} \bigg] \cdot \sup_{\xi \in \mathbb{R}^3, 
  \vert \eta \vert \approx 2^{k_1}} \vert \nu^R_1 (\xi,\eta) \vert 
  \\
& \lesssim 2^{(5/2-)k_1} Z \cdot \big[ 2^{k_2/2} Z + 2^{\delta m} 2^m \rho(2^m) + 2^m \rho(2^m) m^3 \big] 
  \cdot 2^{-2k_1} 2^{10M_0},
\end{align*}
which is an acceptable bound.

\smallskip
\noindent
{\it Subcase 2.1.2: $(\epsilon_1,\epsilon_2) \in \lbrace (+,-) , (-,+) \rbrace$.} 
In this case we note that $\vert \Phi_{\epss} \vert \gtrsim 1.$ 
Therefore we can proceed as for $F_L$ and integrate by parts in time
to produce terms of the type $F^{S}_m$. 
Since the operator $T^{R,2}_1$ satisfies the same bilinear estimates as $T^{S,2}_1,$ 
the bounds are the exact same; we then skip the details and refer 
the reader to the bounds for $T^{S,2}_1$ in the next section. 

\smallskip
{\it Subcase 2.2: $k_2>-m/2 +m/100$ and $k_1 > -m/2 + m/200$.} 
We insert an additional cutoff $\varphi_{k_3}(\eta + \sigma)$
to localize the frequency of the input $G$. We fix $\delta_0 \in (0,1/200)$.

\smallskip
\noindent
{\it Subcase 2.2.1: $k_3 \leqslant -m/2 + \delta_0 m.$}
In this case $\vert \eta \vert \approx \vert \sigma \vert \approx 2^{k_1} \approx 2^{k_2}.$ 
Integrating directly we get, using \eqref{basic-bounds-hardy}, \eqref{nuRest}, 
\begin{align*}
& 2^m \sum_{G,H \in \lbrace g,h \rbrace} \sup_{s \approx 2^m} \Vert  \varphi_k \widehat{\mathcal{F}}_{x \rightarrow \xi} T^{R,2}_1[b_2] (e^{is \epsilon_1 \jxi} \wt{G_{\epsilon_1}}, e^{is \epsilon_2 \jxi} \wt{H_{\epsilon_2}}  \Vert_{L^\infty_{\xi}} \lesssim 2^{m} \cdot 2^{(5/2-)k_3} Z \cdot 2^{(5/2-) k_1}Z \cdot 2^{-2k_1} 2^{10M_0},
\end{align*}
which is acceptable.

\smallskip
\noindent
{\it Subcase 2.2.2: $k_3 >-m/2 + \delta_0 m.$}
Here we integrate by parts in both $\sigma$ and $\eta$ 
after changing variables $\eta \mapsto -\eta-\sigma.$ 
We can then bound, using \eqref{basic-bounds-wh}, \eqref{bounds-g-derL1}, \eqref{bounds-ginfty}, 
\eqref{nuRest} (note that the worse terms correspond to the derivative falling on the profile)
\begin{align*}
& 2^m \sum_{G,H \in \lbrace g,h \rbrace} \sup_{s \approx 2^m} \Vert  \varphi_k \widehat{\mathcal{F}}_{x \rightarrow \xi} T^{R,2}_1[b_2] (e^{is \epsilon_1 \jxi} \wt{G_{\epsilon_1}}, e^{is \epsilon_2 \jxi} \wt{H_{\epsilon_2}}  \Vert_{L^\infty_{\xi}} \\ 
& \lesssim 2^{-m} \big[ 2^{k_3/2} Z + 2^{\delta m} 2^m \rho(2^m) + 2^m \rho(2^m) m^3 \big]
\cdot \big[ 2^{k_2/2} Z + 2^{\delta m} 2^m \rho(2^m) + 2^m \rho(2^m) m^3 \big] \cdot 2^{-2k_1} 2^{10 M_0},
\end{align*}
which is acceptable provided $\delta$, $\delta_0$ and $\delta_N$ are small enough. %, and $N$ large enough.

\smallskip
\noindent
{\bf Step 3: Bounding $M_{1,m}$.}
Taking an inverse Fourier transform of $M_{1,m}$ and relying on the dispersive estimates \eqref{basic-bounds-decay},
we have
\begin{align*}
\big \Vert M_{1,m} \big \Vert_{L^{\infty}_{\xi}} 
  & \leqslant \bigg \Vert \int_0^t  B(s) \, \big( L^{-1} e^{isL} f\big)  \phi 
  \, \tau_m(s) ds \bigg \Vert_{L^1_x} \lesssim 2^m \sup_{s \approx 2^m} \vert B(s) \vert  
  \Vert L^{-1}  e^{isL} f \Vert_{L^6_x} \lesssim 2^m \cdot  2^{-m/2} \cdot 2^{-m + 10 \delta_N} Z  ,
\end{align*}
which yields the desired result.

\smallskip
\noindent
{\bf Step 4: Bounding $F_H$.}
From the definition of $F_H = \sum_{m} F_{H,m}$ given through \eqref{FHL}, 
we see that when we express this term in distorted Fourier space at least one of the three frequencies 
(the output $\xi$, and the two input frequencies $\eta,\s$) 
is of size $\gtrsim 2^{M_0(s)} \approx \js^{\delta_N}$. 
Then we can split $F_{H,m}$ into seven terms by inserting frequency projections $P_+$ and $P_-$
associated to cutoffs $\varphi_{+} := \varphi_{\leqslant M_0}$ and $\varphi_{-} := \varphi_{>M_0}$
%where we recall that $M_0=M_0(s) \approx \delta_N \log \js$; 
(see \eqref{FHL}).
We denote these terms by $F_{H,m}^{(i,j,k)}$ where the signs $i, j, k \in \lbrace +,- \rbrace$
correspond to the presence of a cutoff $\varphi_+$ or $\varphi_-$, and $(i,j,k) \neq (+,+,+)$.

\smallskip
{\it Case 1: $j= -$ or $k=-$.} 
Without loss of generality we let $j=-$.
Then we write, for $k\in\{+,-\}$ and using \eqref{basic-bounds-eh},
\begin{align*}
\Vert F_{H,m}^{(i,-,k)}  \Vert_{L^{\infty}_{\xi}} & \lesssim \Bigg \Vert \widehat{\mathcal{F}}^{-1}_{\xi \rightarrow x} \Bigg[ \varphi_{i}(\xi) \int_0^t \int_{\mathbb{R}^6} \frac{e^{-is\Phi_{\epsilon_1,\epsilon_2}(\xi,\eta,\sigma)}}{\jeta \jsigma} \varphi_{-}(\eta) \wt{f_{\epsilon_1}}(s,\eta) \varphi_{k}(\sigma) \wt{f_{\epsilon_2}}(s,\sigma) \mu(\xi,\eta,\sigma) d\eta d \sigma \, \tau_m(s) ds \Bigg] \Bigg \Vert_{L^1_x} \\
& \lesssim 2^m \bigg \Vert L^{-1} e^{itL} P_{-} f_{\epsilon_1} 
  \cdot  L^{-1} e^{itL} P_k f_{\epsilon_2} \bigg \Vert_{L^1_x} 
  \lesssim 2^m \cdot 2^{-N M_0} \Vert f_{\epsilon_1} \Vert_{H^N} \Vert f_{\epsilon_2} \Vert_{L^2},
\end{align*}
which is more than sufficient since $M_0 = \delta_N m = \frac{5}{N} m.$ 
%The reasoning when $k=-$ is identical, putting the profile $f_{\epsilon_2}$ in $H^N$ instead of $L^2.$

\smallskip
{\it Case 2: $(i,j,k)=(-,+,+).$} 
In this case using the boundedness of wave operators on $L^1$ and distributing derivatives,
we can estimate 
\begin{align*}
& \Vert F_{H,m}^{(-,+,+)}  \Vert_{L^{\infty}_{\xi}}
\\
& \lesssim 2^{-N M_0} \bigg \Vert \widehat{\mathcal{F}}^{-1}_{\xi \rightarrow x} \vert 
  \xi \vert^N \varphi_{i}(\xi) \int_0^t \int_{R^6} 
  \frac{e^{-is\Phi_{\epsilon_1,\epsilon_2}(\xi,\eta,\sigma)}}{\jeta \jsigma} \varphi_{-}(\eta) 
  \wt{f_{\epsilon_1}}(s,\eta) \varphi_{k}(\sigma) \wt{f_{\epsilon_2}}(s,\sigma) \mu(\xi,\eta,\sigma) 
  d\eta d \sigma \, \tau_m(s) ds \bigg \Vert_{L^1_x} 
  \\
& \lesssim 2^{-NM_0} \sum_{N_1 + N_2 \leqslant N} 
  \big \Vert D^{N_1} L^{-1} e^{itL} f_{\epsilon_1} \big \Vert_{L^2_x} 
  \big \Vert D^{N_2}  L^{-1} e^{itL} f_{\epsilon_2} \big \Vert_{L^2_x};
\end{align*}
we then conclude as we did above using \eqref{basic-bounds-eh}.

\smallskip
\noindent
{\bf Step 5: Bounding $S_m.$} 
Recall the definitions \eqref{S}.
To treat $S_{2,m}$ we integrate by parts in time and 
then use the estimate for $\dot{B}$ in \eqref{A-B} to obtain 
\begin{align*}
%&\bigg \Vert (1- \chi_{C}(\xi)) \int_0^t \tau_m(s) \, e^{-is(\langle \xi \rangle -2 \lambda)} B^2(s) 
%\,  \widetilde{\theta(\xi)} ds \bigg \Vert_{L^{\infty}_{\xi}} 
{\| S_{2,m}(t) \|}_{L^\infty_\xi}
& \lesssim \bigg \Vert \tau_m(t)  B^2(t) \frac{e^{-it(\langle \xi \rangle - 2 \lambda)}}{\langle \xi \rangle -2 \lambda} 
(1-\chi_C(\xi)) \, \widetilde{\theta(\xi)} \bigg \Vert_{L^{\infty}_{\xi}} 
\\
& + \bigg \Vert  \frac{1-\chi_C(\xi)}{\langle \xi \rangle - 2 \lambda} 
 \int_0^t e^{-is(\langle \xi \rangle -2\lambda)} 
 \big( 2 \dot{B}(s) B(s) \tau_m(s) + B^2(s) 2^{-m} \tau'_m(s) %+ s^{-1} B^2(s)\tau_m(s)
 \big) ds \, \widetilde{\theta(\xi)} \bigg \Vert_{L^2} 
  \\
& \lesssim \rho(2^m) + \int_{s \approx 2^m} %s \rho(s)^2
  s^{-1} \rho(s) ds \lesssim \rho(2^m).
\end{align*}
For the term $S_{1,m}$ we note that $A^2-B^2 = O(A^3),$ see \eqref{A-B},
and therefore we can bound directly
%\begin{align*}
$\Vert S_{1,m} \Vert_{L^{\infty}_{\xi}} \lesssim 
  %\int_{0}^t O\big( |A|^3(s) \big) \tau_m(s) \,ds \lesssim 
  2^m \rho(2^m)^{3/2}$, which suffices.
%\end{align*}
%\end{proof}
The proof of \eqref{scatteringh} is complete. $\hfill \Box$

\section{Estimating the correction}\label{seccorr}
We conclude the proof by estimating the normal form correction defined in the previous section, 
namely showing \eqref{corrh}, under the a priori assumption \eqref{bootLinfty}.
Recalling \eqref{defN} and \eqref{Sbdry}, we have to estimate $T_1^{S,1}, T_1^{S,2}$ and $T_{i}^S, i = 2,3.$ 

\noindent
\textbf{Step 1: Bounding $T_{1}^{S,1}.$} 
Using the identities \eqref{nu1S}, \eqref{nu0} and \eqref{estimatesg}
we can reduce matters to estimating the following three operators:
%start by decomposing, using identities \eqref{nu1S}, \eqref{nu0} and \eqref{estimatesg}:
\begin{align*}
%T_{1}^{S,1} &:= T_{\delta}^{S,1} + T_{p.v.}^{S,1} + T_{b}^{S,1} , \\
T_{\delta}^{S,1}[b](G,H) &:= %-2 \pi^2 
  \whF^{-1}_{\xi\rightarrow x} \iint_{\R^3\times\R^3} \wt{G}(\xi-\eta) \wt{H}(\sigma) 
  \,b(\xi,\eta,\sigma)\, \frac{b_0(%\eta/\vert \eta \vert
  \hat{\eta}, \sigma)}{\vert \eta \vert} \delta (\vert \eta \vert - \vert \sigma \vert) 
  %\varphi_{\leqslant -M_0-5}(\vert \eta \vert - \vert \sigma \vert) 
  \, d\eta d\sigma,
\\
T_{p.v.}^{S,1}[b](G,H) &:= %2i \pi 
  \whF^{-1}_{\xi\rightarrow x} \iint_{\R^3\times\R^3} \wt{G}(\xi-\eta) \wt{H}(\sigma) 
  \,b(\xi,\eta,\sigma)\,  \frac{b_0(\hat{\eta}%-\eta/\vert \eta \vert
  , \sigma)}{\vert \eta \vert} 
  \pv \frac{ \varphi_{\leqslant -M_0-5}(\vert \eta \vert - \vert \sigma \vert)}{
  \vert \eta \vert - \vert \sigma \vert} \, d\eta d\sigma,
\end{align*}
where we denoted $\hat{\eta} := -\eta/|\eta|$, and, for $a=1,\dots,N_2$,
\begin{align*}
T_{a}^{S,1}[b](G,H) := %2i \pi 
  \whF^{-1}_{\xi\rightarrow x} \iint_{\R^3\times\R^3} G(\xi-\eta) H(\sigma) 
  \,b(\xi,\eta,\sigma)\,  \varphi_{\leqslant -M_0-5}(\vert \eta \vert - \vert \sigma \vert)
  \frac{1 %b_0(-\eta/\vert \eta \vert, \sigma)
  }{\vert \eta \vert} 
  \\
  \times %\sum_{a=1}^{N_2} 
  \sum_{J \in \mathbb{Z}} b_{a,J}(\eta,\sigma) \cdot 2^J K_a(\vert \eta \vert - \vert \sigma \vert) \, d\eta d\sigma.  
\end{align*}
Recall the definition of $b$ in \eqref{Sintb}. %\cfp{Fix inconsistencies}
Notice that in the formulas above the inputs $(G,H)$ play the role of 
either $e^{\pm itL}h$ or $e^{\pm itL}g$.
In particular the $H^s_x$ and $\wtF^{-1}L^\infty_\xi$ norms of $(G,H)$ are those of $g$ and $h$.

We can deal with $T_{\delta}^{S,1}, T_{a}^{S,1}$ by direct integration, 
using the lower bound $\vert \Phi(\xi,\eta,\sigma) \vert \gtrsim \big(\jxi + \jeta + \jsigma \big)^{-1}$ 
see (5.47), Lemma 5.10 in \cite{LP}. More precisely,
\begin{align*}
\big \Vert T_{\delta}^{S,1}[b](G,H) \big \Vert_{L^{\infty}_{\xi}} 
  & \lesssim \Bigg \Vert \frac{\varphi_{k_1}(\eta) %\fp{\varphi_{k_1}(\sigma)}
  }{\Phi(\xi,\xi-\eta,\eta%\sigma
  )
  %\jxi 
  \langle \xi-\eta \rangle \jeta} \Bigg \Vert_{L^{\infty}_{\xi,\eta%,\sigma
  }} 
  \Vert b_0(\omega,\xi) \Vert_{L^{\infty}_{\omega,\xi}} 
  \\
& \times \sup_{\eta \in \mathbb{R}^3} \Bigg \vert \int_{\mathbb{R}^3} 
  \wt{H}(\sigma) \varphi_{k_2}(\sigma) 
  \frac{\delta(\vert \eta \vert - \vert \sigma \vert)}{\vert \eta \vert} d\sigma \Bigg \vert 
  \cdot \big \Vert %\jxi 
  G(\xi-\eta) \varphi_{k_1}(\eta) \big \Vert_{L^{\infty}_{\xi} L^1_{\eta}} 
  \\
& \lesssim \langle 2^{k_1} \rangle^{-1} %\langle 2^{k_2} \rangle^{-1} 
  \cdot 2^{-k_1} {\| \wt{H} \|}_{L^{\infty}_{\sigma}} 
  \cdot 2^{3k_1/2} %\langle 2^{k_1} \rangle \Vert H \Vert_{H^2} 
  {\| G \|}_{L^2} 
  \lesssim 2^{k_1/2}  \langle 2^{k_1} \rangle^{-1} \e^{1-}.
\end{align*}
We can then sum over $k_1$ and get a bound consistent with \eqref{corrh}.
Note that we have used \eqref{basic-bounds-eh}. %k _2,$ relying on the condition $\vert k_1 - k_2 \vert <5.$ 
Similarly we can bound
\begin{align*}
\big \Vert T_{a}^{S,1}[b](G,H) \big \Vert_{L^{\infty}_{\xi}} 
  & \lesssim \Bigg \Vert \frac{\varphi_{k_1}(\eta) \varphi_{k_2}(\sigma)}{\Phi(\xi,\xi-\eta,\sigma)
  \jxi \langle \xi-\eta \rangle \jsigma} \Bigg \Vert_{L^{\infty}_{\xi,\eta,\sigma}} 
  \Vert b_0(\omega,\xi) \Vert_{L^{\infty}_{\omega,\xi}} 
  \cdot \big \Vert \jxi \wt{G}(\xi-\eta) \varphi_{k_1}(\eta) \big \Vert_{L^1_{\eta}} 
  \\
& \times \sup_{\xi,\eta \in \mathbb{R}^3 } %\sum_{a=1}^{N_2} 
  \sum_{J \in \mathbb{Z}}  
  \Bigg \vert \varphi_{\sim k_1}(\eta) \int_{\mathbb{R}^3} \wt{H}(\sigma)\varphi_{\sim k_2}(\sigma) 
  b_{a,J}(\eta,\sigma) \cdot 2^J K_a\big(2^J(\vert \eta \vert - \vert \sigma \vert ) \big) d\sigma \Bigg \vert  
  \\
& \lesssim \langle 2^{k_1} \rangle^{-1} \langle 2^{k_2} \rangle^{-1} 
  \Vert \wt{H} \Vert_{L^{\infty}_{\sigma}} \cdot 
 2^{k_1/2} \langle 2^{k_1} \rangle
  \Vert G \Vert_{H^2}
  \cdot \sum_{J \in \mathbb{Z}}
  \big \vert \varphi_{\sim k_1} (\eta) \varphi_{\sim k_2} (\sigma) b_{a,J}(\eta,\sigma) 
  \big \vert
  \\
  & \lesssim 2^{k_1/2} \langle 2^{k_2} \rangle^{-1} \varepsilon^{1-}.
\end{align*}
Note that here we have used again \eqref{basic-bounds-eh},
and also the a apriori assumption \eqref{bootLinfty} and Lemma \ref{gLinftyO(1)}
to bound $\wt{H}$ in $L^\infty_\s$.
We can then sum over $k_1,k_2$ using $\vert k_1 - k_2 \vert <5$, see \eqref{nu1S}.

We then look at the principal value part and bound %(we omit the $\pv$ notation)
\begin{align*}
& \big| T_{p.v.}^{S,1}[b](G,H)(\xi) \big|
  %& = \bigg \vert \int_{\mathbb{R}^3} \wt{G}(\xi-\eta) \int_{\mathbb{R}^3} 
  %\frac{1}{\vert \eta \vert}\frac{\wt{g}(-\eta /\vert \eta \vert, \sigma)}{\vert \eta \vert - \vert \sigma \vert}  
  %\frac{ \varphi_{k_1}(\eta) \varphi_{k_2}(\sigma)}{\Phi_{\epss}(\xi,\xi-\eta,\sigma) \langle \xi-\eta \rangle \jsigma}  
  %\wt{H}(\sigma)  d\sigma \, d\eta \bigg \vert 
  \leqslant |I(\xi)| + |II(\xi)|, 
  \\
& I(\xi) := \int_{\mathbb{R}^3} \frac{\wt{G}(\xi-\eta)}{\langle \xi-\eta \rangle} 
  \frac{\varphi_{k_1}(\eta)}{\vert \eta \vert} \int_{\mathbb{R}^3} 
  \frac{b_0(\hat{\eta},\s)}{\Phi_{\epss}(\xi,\xi-\eta, \eta)
  \Phi_{\epss}(\xi,\xi-\eta,\sigma)}
  \\
& \qquad \qquad \qquad \times \wt{H}(\sigma) \frac{ \Phi_{\epss}(\xi,\xi-\eta,\sigma) 
  - \Phi_{\epss}(\xi,\xi-\eta,\eta)}{\vert \eta \vert-\vert \sigma \vert}
  \frac{\varphi_{k_2}(\sigma)}{\jsigma}  d\sigma d\eta,
\\  
& II(\xi) := \int_{\mathbb{R}^3} \frac{\wt{G}(\xi-\eta)}{\langle \xi-\eta \rangle} 
  \frac{\varphi_{k_1}(\eta)}{\Phi_{\epss}(\xi,\xi-\eta,\eta) \vert \eta \vert} 
  \, \pv \int_{\mathbb{R}^3} \frac{b_0(\hat{\eta},\s)}{\vert \eta \vert - \vert \sigma \vert} 
  \frac{\wt{H}(\sigma)}{\langle \sigma \rangle} \varphi_{k_2}(\sigma)  d\sigma d\eta.
\end{align*}
The term $I$ can be treated using direct integration as $T_{\delta}^{S,1}$ and $T_{a}^{S,1}$
since the singularity is canceled by the difference of the phases.
For $II$ we use \eqref{newbilin1} and the fact that 
%\begin{align*}
%\bigg \Vert \frac{1}{
$|\Phi_{\epss}(\xi,\xi-\eta,\eta) \langle \xi-\eta \rangle \jeta |
%}\bigg \Vert_{L^{\infty}_{\xi,\eta}} \lesssim 1,
 \gtrsim 1,$
%\end{align*}
%which follows from the lower bound $\vert \Phi_{\epss} (\xi,\xi-\eta,\eta) 
%\vert \gtrsim \big(\jxi + \langle \xi-\eta \rangle +\jeta \big)^{-1},$ 
and find 
\begin{align*}
\Vert II \Vert_{L^{\infty}_{\xi}} 
  \lesssim %2^{k_1} \langle 2^{k_1} \rangle^{-1} \cdot 
  \langle 2^{k_1} \rangle \cdot 2^{k_2} \langle 2^{k_2} \rangle^{-3} (\varepsilon^{1-})^2.
\end{align*}

\smallskip
\noindent
\textbf{Step 2: Bounding $T_1^{S,2}.$} 
We can use similar argument and eventually rely on \eqref{newbilin2} for the hardest term. 
Defining $T_{\delta}^{S,2}, T_{p.v.}^{S,2}$ and $T_{a}^{S,2}$ in an analogous way 
to what was done above, i.e., replacing the measure $\mu_1^S$ in the expression $T_{1}^{S,2}$ 
(see \eqref{idop12}) by the three pieces in \eqref{nu1S}, \eqref{nu0}.

We only focus on the principal value part since the other terms can be bounded 
by direct integration as above.  We start by writing that
\begin{align*}
\big| T_{p.v.}^{S,2}[b](G,H) \big| %_{L^{\infty}_{\xi}} 
  & \leqslant |III(\xi)| + |IV(\xi)|, 
\\
III(\xi) & := \iint_{\mathbb{R}^6} \frac{\wt{G}(\eta + \sigma)}{\langle \eta + \sigma \rangle}  
  \frac{\varphi_{k_1}(\eta) \varphi_{k_2}(\sigma) }{\Phi_{\epss}(\eta,\eta + \sigma, \sigma)}
  \frac{\wt{H}(\sigma)}{\jsigma}  \frac{b_0(\hat{\eta}%-\eta/\vert \eta \vert
  ,\xi)}{\vert \eta \vert} 
  \pv \frac{1}{\vert \eta \vert - \vert \xi \vert} d\eta d\sigma , 
  \\
IV(\xi) & := \iint_{\mathbb{R}^6} \frac{\wt{G}(\eta + \sigma)}{\langle \eta + \sigma \rangle}  
  \varphi_{k_1}(\eta) \varphi_{k_2}(\sigma) \frac{\wt{H}(\sigma)}{\jsigma}  
  \frac{b_0(\hat{\eta}%-\eta/\vert \eta \vert
  ,\xi)}{\vert \eta \vert} 
  \\
&\times \frac{\Phi_{\epss}(\eta,\eta + \sigma, \sigma) 
  -  \Phi_{\epss}(\xi,\eta + \sigma, \sigma)}{\Phi_{\epss}(\eta,\eta + \sigma, \sigma) 
  \Phi_{\epss}(\xi,\eta + \sigma, \sigma)}  \pv \frac{1}{\vert \eta \vert - \vert \xi \vert} d\eta d\sigma .
\end{align*}
Since the term $IV$ is not singular we can treat it by direct integration as above; we skip the details. 
To bound the main term $III$ we apply \eqref{newbilin2} to the multiplier 
\begin{align*}
m(\eta,\sigma) := \frac{\varphi_{k_1}(\eta) \varphi_{k_2}(\sigma)}{\Phi_{\epss}(\eta,\eta + \sigma,\sigma)
  \langle \eta + \sigma \rangle^{5} \langle \sigma \rangle^{5} } \cdot \frac{2^{k_1}}{\vert \eta \vert}, 
\end{align*}
and, using again \eqref{basic-bounds-eh}, we obtain (here $K = k_1$)
\begin{align*}
\Vert III \Vert_{L^{\infty}_{\xi}} & \lesssim  2^{k_1} \langle 2^{k_1} \rangle^{-2} 
  {\| G \|}_{H^6} \Vert P_{\sim k_2} H \Vert_{H^6}  
  \\
  & \lesssim 2^{k_1} \langle 2^{k_1} \rangle^{-2} \e^{1-} %\Vert G \Vert_{H^6} 
  \cdot \min \big\{ \langle 2^{k_2} \rangle^{-1} \Vert  H \Vert_{H^7},  
  2^{3k_2/2} \Vert H \Vert_{L^{\infty}_{\sigma}} \big\}.
\end{align*}
We can then use Lemma \ref{gLinftyO(1)} and \eqref{bootLinfty}
to control the last $L^\infty_\s$ norm and sum over $k_1,k_2$ to conclude.

\smallskip
\noindent
\textbf{Step 3: Bounding $T_{i}^S$.} Using \eqref{mu2SR},  \eqref{nu2S1}, \eqref{nu2S2} and \eqref{mu3S}, 
we can treat these terms by direct integration as we did above for $T_{\delta}^{S,1}.$ 
Estimates are easier in this case, therefore we omit the details.

\end{document}